\documentclass[greek,english,11pt,a4paper,openany]{article}
\usepackage{amssymb,latexsym, amsfonts, amsmath, amsthm}
\usepackage{graphicx}
\usepackage{pgfpages}
\usepackage{enumerate}
\usepackage[backgroundcolor=cyan!20, linecolor=black]{todonotes}
\usepackage{tabularx}
\newcolumntype{Q}{>{$\displaystyle}c<{$}} 
\usepackage{booktabs}
\usepackage[margin=1in]{geometry}
\usepackage{tikz}
\usetikzlibrary{arrows,calc,matrix}

\theoremstyle{definition}
\newtheorem{thm}{Theorem}[section]
\newtheorem*{theorem}{Theorem}

\newtheorem{lemma}[thm]{Lemma}

\newtheorem{remark}[thm]{Remark}

\def\End{\operatorname{End}}
\def\GL{\operatorname{GL}}
\def\Ind{\operatorname{Ind}}

\def\supp{\operatorname{supp}}
\def\Latt{\operatorname{Lat}}
\def\U{\operatorname{U}}
\def\SU{\operatorname{SU}}
\def\SL{\operatorname{SL}}
\def\P{\operatorname{P}}
\def\M{\operatorname{M}}
\def\Hom{\operatorname{Hom}}
\def\Irr{\operatorname{Irr}}

\def\dim{\operatorname{dim}}
\def\Ql{\overline{\mathbb{Q}}_{\ell}}
\def\Zl{\overline{\mathbb{Z}}_{\ell}}
\def\Fl{\overline{\mathbb{F}}_{\ell}}
\def\Gal{\operatorname{Gal}}

\def\St{\operatorname{St}}
\def\adiag{\operatorname{adiag}}
\def\diag{\operatorname{diag}}

\def\Res{\operatorname{Res}}
\def\Ind{\operatorname{Ind}}
\def\GE{{G_E}}
\def\GEi{{G_{E_i}}}

\def\ind{\operatorname{ind}}
\def\I{\operatorname{I}}
\def\R{\operatorname{R}}
\def\presuper#1#2%
  {\mathop{}%
   \mathopen{\vphantom{#2}}^{#1}%
   \kern-\scriptspace%
   #2}
\begin{document}

\title{$\ell$-modular representations of unramified $p$-adic $\U(2,1)$}\author{Robert James Kurinczuk}
\date{\today}
\maketitle
\begin{abstract}
\noindent We construct all cuspidal $\ell$-modular representations of a unitary group in three variables attached to an unramified extension of local fields of odd residual characteristic $p$ with $\ell\neq p$.  We describe the $\ell$-modular principal series and show that the supercuspidal support of an irreducible $\ell$-modular representation is unique up to conjugacy.
\end{abstract}
\section{Introduction}
The abelian category $\mathfrak{R}_R(G)$ of smooth representations of a reductive $p$-adic group $G$ over an algebraically closed field $R$ has been well studied when $R$ has characteristic zero.  The same cannot be said when $R$ has positive characteristic $\ell$; here many questions remain unanswered.  In this paper, we are concerned only in the case $\ell\neq p$. We study the set of isomorphism classes of irreducible $R$-representations $\Irr_R(G)$, eventually specialising to $G=\U(2,1)$ a unitary group in three variables attached to an unramified extension $F/F_0$ of non-archimedean local fields of odd residual characteristic.   All $R$-representations henceforth considered will be smooth.

A classical strategy for the classification of irreducible $R$-representations is to split the problem into two steps: firstly, for any parabolic subgroup $P$ of $G$ with Levi decomposition $P=M\ltimes N$ and any $\sigma\in\Irr_R(M)$, decompose the (normalised) parabolically induced $R$-representation $i^G_P(\sigma)$; and secondly, construct the irreducible $R$-representations which do not appear as a subquotient of a $R$-representation appearing in the first step, the \emph{supercuspidal} $R$-representations.  For any parabolic subgroup $P$, a supercuspidal irreducible $R$-representation $\pi$ will have trivial Jacquet module $r^G_P(\pi)=0$, by Frobenius reciprocity ($i^G_P$ is right adjoint to $r^G_P$).  When $R$ has characteristic zero the irreducible \emph{cuspidal} $R$-representations, those whose Jacquet modules are all trivial, are all supercuspidal.  However in positive characteristic $\ell$, there can exist irreducible cuspidal non-supercuspidal $R$-representations.  

By transitivity of the Jacquet module and the geometric lemma, see Vign\'eras \cite[II 2.19]{Vig96}, the \emph{cuspidal support} of $\pi\in\Irr_R(G)$, that is the set of pairs $(M,\sigma)$ with $M$ a Levi factor of a parabolic subgroup $P$ of $G$ and $\sigma$ an irreducible cuspidal $R$-representation of $M$ such that $\pi$ is a subrepresentation of $i^G_P(\sigma)$, is a non-empty set consisting of a single $G$-conjugacy class; we say that the cuspidal support is unique up to conjugacy.  By transitivity of parabolic induction, the \emph{supercuspidal support} of $\pi\in\Irr_R(G)$, that is the set of pairs $(M,\sigma)$ with $M$ a Levi factor of a parabolic subgroup $P$ of $G$ and $\sigma$ an irreducible supercuspidal $R$-representation of $M$ such that $\pi$ is a subquotient of $i^G_P(\sigma)$, is non-empty.  However, in general, it is not known if the supercuspidal support of $\pi$ is unique up to conjugacy.  

For $\GL_n$ and its inner forms, the supercuspidal support of an irreducible $R$-representation is unique up to conjugacy, due to Vign\'eras \cite{Vig96} and M{\'i}nguez--S\'echerre \cite{VA3,VA2}.   The unicity of supercuspidal support is of great importance.  Firstly, the unicity of supercuspidal support (up to inertia) for $\GL_n$ and its inner forms leads to the block decomposition of $\mathfrak{R}_R(G)$ into indecomposable summands, see S\'echerre--Stevens \cite{SeSt2}.  Secondly, it is important in Vign\'eras' $\ell$-modular local Langlands correspondence for $\GL_n$ which is first defined on supercuspidal elements by compatibility with the characteristic zero local Langlands correspondence and then extended to all irreducible $\ell$-modular representations of $\GL_n$.  In this paper, we prove unicity of supercuspidal support for $\U(2,1)$.  We hope this paper is the first step in establishing similar results for $\U(2,1)$ and in extending these to classical groups in general.

Our strategy is first to construct all irreducible cuspidal $R$-representations by compact induction from irreducible $R$-representations of compact open subgroups.  The type of construction we employ has been used to great effect to construct all irreducible cuspidal $R$-representations in a large class of reductive $p$-adic groups when $R$ has characteristic zero: Morris for level zero $R$-representations of any reductive $p$-adic group \cite{Morris}, Bushnell--Kutzko for $\GL_n$ and $\SL_n$ \cite{BK93,BKSL}, S\'echerre--Stevens for inner forms of $\GL_n$ \cite{SeSt}, Yu--Kim for arbitrary connected reductive groups under ``tame'' conditions \cite{Yutame,Kimtame}, and Stevens for classical $p$-adic groups with $p$ odd \cite{St08}.  Vign\'eras \cite{Vig96} and M{\'i}nguez--S\'echerre, \cite{VA3,VA2}, adapted the characteristic zero constructions for $\GL_n$ and its inner forms to $\ell$-modular representations. We perform similar adaptations to Stevens' construction to exhaust all irreducible cuspidal $\ell$-modular representations of $\U(2,1)$.
 
\begin{theorem}[{Theorem \ref{plcusps}}]
Let $G=\U(2,1)$ and $\pi$ be an irreducible cuspidal $R$-representation of $G$.  Then there exist a compact open subgroup $J$ of $G$ with pro-unipotent radical $J^1$ such that $J/J^1$ is a finite reductive group, an irreducible $R$-representation $\kappa$ of $J$ and an irreducible cuspidal $R$-representation $\sigma$ of $J/J^1$ such that $\pi\simeq \ind_J^G(\kappa\otimes \sigma)$.
\end{theorem}
The construction is explicit and, furthermore, all $R$-representations $$\I_{\kappa}(\sigma)= \ind_J^G(\kappa\otimes \sigma)$$ constructed this way are cuspidal. Moreover, we show that $\I_{\kappa}(\sigma)$ is supercuspidal if and only if $\sigma$ is supercuspidal (Remark \ref{scusplift}). In work in progress, joint with Stevens, we extend Stevens' construction for arbitrary classical groups to the $\ell$-modular setting.

In the split case, for general linear groups all irreducible cuspidal $\ell$-modular representations lift to integral $\ell$-adic representations.  For inner forms of $\GL_n$, this is no longer true; some cuspidal non-supercuspidal $\ell$-modular representations do not lift.  For $\U(2,1)$ we also find cuspidal non-supercuspidal $\ell$-modular representations which do not lift (Remark \ref{notlift}).  These non-lifting phenomena appear quite different. For $\U(2,1)$ this non-lifting occurs because, in certain cases, there are $\ell$-modular representations of the finite group $J/J^1$ which do not lift.  For inner forms of $\GL_n$, the non lifting occurs when the normaliser of the reduction modulo-$\ell$ of the inflation of a cuspidal $\ell$-adic representation of an analogous group to $J/J^1$ is larger than the normaliser of all of its cuspidal lifts.  We find that all supercuspidal $\ell$-modular representations of $\U(2,1)$ lift (Remark \ref{scusplift}), as is the case for $\GL_n$ and its inner forms.

Secondly, by studying corresponding Hecke algebras, we find the characters $\chi$ of the maximal diagonal torus $T$ of $\U(2,1)$ such that the principal series $R$-representation $i^{\U(2,1)}_B(\chi)$ is reducible.  We let $\chi_1$ denote the character of $F^\times$ defined by $\chi_1(x)=\chi(\diag(x,\overline{x}x^{-1},\overline{x}^{-1}))$ where $\overline{x}$ is the $\Gal(F/F_0)$-conjugate of $x$.

\begin{theorem}[{Theorem \ref{redpoints}}]
Let $G=\U(2,1)$.  Then $i^G_B(\chi)$ is reducible exactly in the following cases:
\begin{enumerate}
\item $\chi=\delta_B^{\pm \frac{1}{2}}$ where $\delta_B$ is the modulus character;
\item $\chi= \eta\delta_B^{\pm \frac{1}{4}}$ where $\eta$ is any extension of the quadratic class field character $\omega_{F/F_0}$ to $F^\times$;
\item $\chi_1$ is nontrivial, but $\chi_1\mid_{F_0^\times}$ is trivial.
\end{enumerate}
\end{theorem}

When $R$ is of characteristic zero this is due to Keys \cite{keys}.  In our proof we need to apply the results of Keys to determine a sign.  It should be possible to remove this dependancy by computation using the theory of covers (\emph{cf.} Blondel \cite[Remark 3.13]{Blarxiv}).  An alternative proof, when $F_0$ is of characteristic zero, would be to use the computations of Keys with Dat \cite[Proposition $8.4$]{Dat}.  

Finally, by studying the interaction of the right adjoints $\R_{\kappa}$ of the functors $\I_{\kappa}$ with parabolic induction we find cuspidal subquotients of the principal series. When cuspidal subquotients appear in the principal series we show exactly which ones from our exhaustive list do; finding that the supercuspidal support of an irreducible $R$-representation is unique up to conjugacy. 

\begin{theorem}[Theorem \ref{scuspsupp}]
Let $\pi$ be an irreducible $R$-representation of $\U(2,1)$.  The supercuspidal support of $\pi$ is unique up to conjugacy.
\end{theorem}

In fact, in many cases, we obtain extra information on the irreducible quotients and subrepresentations which appear.  If $\ell\neq 2$ and $\ell\mid q-1$ we show that all the principal series $R$-representations $i^G_B(\chi)$ are semisimple (Lemma \ref{semisimple}).  If $\ell\mid q+1$ we show that $i^G_B(\chi)$ has a unique irreducible subrepresentation and a unique irreducible quotient and these are isomorphic (Lemma \ref{lemma611}).  A striking example of the reducibilities that occur is when $\chi=\delta_B^{-\frac{1}{2}}$.

\begin{theorem}[{see Theorem \ref{urprincipaldecomp} for more details}]
Let $G=\U(2,1)$.  
\begin{enumerate}
\item If $\ell\nmid (q-1)(q+1)(q^2-q+1)$ then $i^G_B(\delta_B^{-\frac{1}{2}})$ has length two with unique irreducible subrepresentation $1_G$ and unique irreducible quotient $\St_G$.
\item If $\ell\neq 2$ and $\ell\mid q-1$ then $i^G_B(\delta_B^{-\frac{1}{2}})=1_G\oplus \St_G$ is semisimple of length two.
\item If $\ell\neq 3$ and $\ell\mid q^2-q+1$ then $i^G_B(\delta_B^{-\frac{1}{2}})$ has length three with unique cuspidal subquotient. The unique irreducible subrepresentation is not isomorphic to the unique irreducible quotient.
\item If $\ell\neq 2$ and $\ell\mid q+1$ or if $\ell=2$ and $4\mid q+1$,  then $i^G_B(\delta_B^{-\frac{1}{2}})$ has length six with $1_G$ appearing as the unique subrepresentation and the unique quotient and four cuspidal subquotients, one of which appears with multiplicity two. A maximal cuspidal subquotient of $i^G_B(\delta_B^{-\frac{1}{2}})$ is not semisimple.  
\item If $\ell=2$ and $4\mid q-1$, then $i^G_B(\delta_B^{-\frac{1}{2}})$ has length five with $1_G$ appearing as the unique subrepresentation and the unique quotient.  All cuspidal subquotients of $i^G_B(\delta_B^{-\frac{1}{2}})$ are semisimple and the irreducible cuspidal subquotients are pairwise non-isomorphic.
\end{enumerate}
\end{theorem}

\subsubsection*{Acknowledgements} 
Part of this work was included in my doctoral thesis.  I would like to thank my supervisor, Shaun Stevens, for his support and encouragement both during my doctorate and after.  I would like to thank Vincent S\'echerre and Alberto M\'inguez for many useful conversations.  Particularly, I would like to thank Vincent S\'echerre for some technical suggestions regarding $\kappa$-induction and $\kappa$-restriction. 

\section{Notation}

\subsection{Unramified unitary groups} Let $F_0$ be a non-archimedean local field of odd residual characteristic.  Let $F$ be an unramified quadratic extension of $F_0$ and $\overline{\phantom{w}}$ a generator of $\Gal(F/F_0)$.  If $D$ is a non-archimedean local field, we let $\mathfrak{o}_D$ denote the ring of integers of $D$, $\mathfrak{p}_D$ denote the unique maximal ideal of $\mathfrak{o}_D$, and $k_D=\mathfrak{o}_D/\mathfrak{p}_D$ denote the residue field. We let $\mathfrak{o}_0=\mathfrak{o}_{F_0}$, $\mathfrak{p}_0=\mathfrak{p}_{F_0}$, and $k_0=k_{F_0}$.  We fix a choice of uniformiser $\varpi_F$ of $F_0$.

 Let $V$ be a finite dimensional $F$-vector space and $h:V\times V\rightarrow F$ a hermitian form on $V$, that is a nondegenerate form which is linear in the first variable, $\overline{\phantom{w}}$--linear in the second variable and such that, for all $v_1,v_2\in V$, $h(v_1,v_2)=\overline{h(v_2,v_1)}$.  The \emph{unitary group} $\U(V,h)$ is the subgroup of isometries of $\GL(V)$, i.e. $\U(V,h)=\{g\in\GL(V):h(gv_1,gv_2)=h(v_1,v_2),~v_1,v_2\in V\}$.  The form $h$ induces an anti-involution on $\End_F(V)$ which we denote by $(\phantom{w})^\sigma$.
 
 \subsection{Parahoric subgroups}
An \emph{$\mathfrak{o}_F$-lattice} in $V$ is a compact open $\mathfrak{o}_F$-submodule of $V$.  Let $L$ be an $\mathfrak{o}_F$-lattice in $V$ and let $\Latt V=\{\mathfrak{o}_F\text{-lattices in }V\}$.  The $\mathfrak{o}_F$-lattice $L^\sharp=\{v\in V:h(v,L)\subseteq \mathfrak{p}_F\}$, defined relative to $h$, is called the dual lattice of $L$.   Let $A=\End_F(V)$ and $\mathfrak{g}=\{X\in A: X+X^{\sigma}=0\}$.  
An \emph{$\mathfrak{o}_F$-lattice sequence} is a function $\Lambda:\mathbb{Z}\rightarrow \Latt V$ which is decreasing and periodic. Let $\Lambda$ be an $\mathfrak{o}_F$-lattice sequence.  The \emph{dual} $\mathfrak{o}_F$-lattice sequence $\Lambda^\sharp$ of $\Lambda$ is the $\mathfrak{o}_F$-lattice sequence defined by, for all $n\in\mathbb{Z}$, $\Lambda^\sharp(n)=(\Lambda(-n))^\sharp$.  We call $\Lambda$ \emph{self dual} if there exists $k\in\mathbb{Z}$ such that $\Lambda(n)=\Lambda^\sharp(n+k)$,  for all $n\in\mathbb{Z}$.  If $\Lambda$ is self dual then we can always consider a translate $\Lambda_k$ of $\Lambda$ such that either $\Lambda_k(0)=\Lambda_k^{\sharp}(0)$ or $\Lambda_k(1)=\Lambda_k^{\sharp}(0)$.

Let $\Lambda$ be an $\mathfrak{o}_F$-lattice sequence in $V$.  For $n\in\mathbb{Z}$ define
$$\mathfrak{P}_n(\Lambda)=\{x\in A:x\Lambda(m)\subset \Lambda(m+n), \text{ for all }m\in\mathbb{Z}\},$$
which is an $\mathfrak{o}_F$-lattice in $A$.  We let $\mathfrak{P}^{-}_n(\Lambda)=\mathfrak{P}_n(\Lambda)\cap \mathfrak{g}$.  

If $\Lambda$ is self dual then the groups $\mathfrak{P}_n(\Lambda)$ are stable under the anti--involution which $h$ induces on $A$.  In this case, define compact open subgroups of $G$ called \emph{parahoric subgroups}, by
\[\P(\Lambda)=\mathfrak{P}_0(\Lambda)^\times \cap G;\]
and
\[\P_m(\Lambda)=(1+\mathfrak{P}_m(\Lambda))\cap G, ~m\in\mathbb{N}.\]
The pro-unipotent radical of $\P(\Lambda)$ is isomorphic to $\P_1(\Lambda)$. The sequence $(\P_m(\Lambda))_{m\in\mathbb{N}}$ is a fundamental system of neighbourhoods of the identity in $G$ and forms a decreasing filtration of $\P(\Lambda)$ by normal compact open subgroups. The quotient $\M(\Lambda)=\P(\Lambda)/\P_1(\Lambda)$ is the $k_0$-points of a connected reductive group defined over $k_0$.

Let $\P_1=\P(\Lambda_1)$ and $\P_2=\P(\Lambda_2)$ be parahoric subgroups of $G$.  Fix a set of \emph{distinguished} double coset representatives $D_{2,1}$ for $\P_2\backslash G/\P_1$, as in Morris \cite[\S 3.10]{Morris93}.  Let $n\in D_{2,1}$ then 
\begin{align*}P_{\Lambda_1,n\Lambda_2}=\P_1^1(\P_1\cap \P_2^n)/\P_1^1\end{align*}
is a parabolic subgroup of $\M_1=\P_1/\P_1^1$, by Morris \cite[Corollary 3.20]{Morris93}.  Furthermore, the pro-$p$ unipotent radical of $\P_1^1(\P_1\cap \P_2^n)$ is $\P_1^1(\P_1\cap (\P_2^n)^1)$, by Morris \cite[Lemma 3.21]{Morris93}.  If $D_{2,1}$ is a set of distinguished double coset representatives for $\P_2\backslash G/\P_1$ then $D_{2,1}^{-1}$ is a set of distinguished double coset representatives for $\P_1\backslash G/\P_2$.  Hence
\begin{align*}
P_{\Lambda_2,n^{-1}\Lambda_1}=\P_2^1(\P_2\cap \presuper{n}\P_1)/\P_2^1
\end{align*}
is a parabolic subgroup of $\M_2=\P_2/\P_2^1$.  Furthermore, the pro-$p$ unipotent radical of $\P_2^1(\P_2\cap \presuper{n}\P_1)$ is $\P_2^1(\P_2\cap \presuper{n}\P_1^1)$.  

\subsection{$\U(2,1)(F/F_0)$}
Let $x_{i}\in F$ for $i=1,2\ldots, n$.  Denote by $\diag(x_1,x_2,\ldots, x_n)$ the $n$ by $n$ diagonal matrix with entries $x_i$ on the diagonal and by $\adiag(x_1,x_2,\ldots,x_n)$ the $n$ by $n$ matrix  $(a_{i,j})$ such that $a_{m,n+1-m}=x_{n+1-m}$ and all other entries are zero.

Let $V$ be a three dimensional $F$-vector space with standard basis $\{e_{-1},e_0,e_1\}$ and $h:V\times V\rightarrow F$ be the non-degenerate hermitian form on $V$ defined by, let $v,u\in V$,
\begin{align*}h(v,w)=v_{-1}\overline{u_1}+v_0\overline{u_0}+v_1\overline{u_{-1}}\end{align*}
if $v=(v_{-1},v_0,v_1)$ and $w=(u_{-1},u_0,u_1)$ with respect to the standard basis $\{e_{-1},e_0,e_1\}$.  Let $\U(2,1)(F/F_0)$ denote the unitary group attached to the hermitian space $(V,h)$, i.e.
\begin{align*}
\U(2,1)(F/F_0)=\{g\in \GL_3(F): gJ\overline{g}^{T}J=1\} 
\end{align*}
where $J=\adiag(1,1,1)$ is the matrix of the form $h$.  We let $\U(1,1)(F/F_0)$ and $\U(2)(F/F_0)$ denote the two dimensional unitary groups defined by the forms with associated matrices are $\adiag(1,1)$ and $\diag(1,\varpi_F)$ respectively. Let $\U(1)(F/F_0)=\{g\in F^\times: g\overline{g}=1\}$ and occasionally, for brevity, let $F^1=\U(1)(F/F_0)$.  We use analogous notation for unitary groups defined over extensions of $F_0$ and defined over finite fields.

Let $B$ be the standard Borel subgroup of $\U(2,1)(F/F_0)$  with Levi decomposition $B=T\ltimes N$ where $T=\{\diag(x,y,\overline{x}^{-1}):x\in F^\times, y\in F^1\}$ and 
\begin{align*}
N=\left\{\begin{pmatrix} 1&x&y\\0&1&\overline{x}\\0&0&1\end{pmatrix}:x,y\in F, y+\overline{y}=x\overline{x}\right\}.
\end{align*} 
The maximal $F_0$-split torus $T_0$ contained in $T$ is $T_0=\{\diag(x,1,x^{-1}):x\in F_0\}$.  The subgroup of $T$ generated by all of its compact subgroups is $T^0=\{\diag(x,y,\overline{x}^{-1}):x\in\mathfrak{o}_F^\times,y\in F^1\}$.  Let $T^1=T^0\cap \diag(1+\mathfrak{p}_F,1+\mathfrak{p}_F,1+\mathfrak{p}_F)$.
 
Let $\Lambda_I$ be the $\mathfrak{o}_F$-lattice sequence of period three given by $\Lambda_I(0)=\mathfrak{o}_F\oplus\mathfrak{o}_F\oplus\mathfrak{o}_F$, $\Lambda_I(1)=\mathfrak{o}_F\oplus\mathfrak{o}_F\oplus\mathfrak{p}_F$ and $\Lambda_I(2)=\mathfrak{o}_F\oplus\mathfrak{p}_F\oplus\mathfrak{p}_F$ with respect to the standard basis.  The (standard) Iwahori subgroup of $G$ is the parahoric subgroup
\begin{align*}\P(\Lambda_I)=\begin{pmatrix}\mathfrak{o}_F&\mathfrak{o}_F&\mathfrak{o}_F\\\mathfrak{p}_F&\mathfrak{o}_F&\mathfrak{o}_F\\\mathfrak{p}_F&\mathfrak{p}_F&\mathfrak{p}_F\end{pmatrix}\cap G.\end{align*} 
There are two parahoric subgroups of $G$ which contain $\P(\Lambda_I)$, both of which are maximal.  These correspond to the lattice sequences $\Lambda_x$ of period one with $\Lambda_x(0)=\mathfrak{o}_F\oplus\mathfrak{o}_F\oplus\mathfrak{o}_F$; and $\Lambda_y$ of period two with $\Lambda_y(0)=\mathfrak{o}_F\oplus\mathfrak{o}_F\oplus\mathfrak{p}_F$ and $\Lambda_y(1)=\mathfrak{o}_F\oplus\mathfrak{p}_F\oplus\mathfrak{p}_F$.  We have $\M(\Lambda_x)\simeq \U(2,1)(k_F/k_0)$, $\M(\Lambda_y)\simeq \U(1,1)(k_F/k_0)\times k_F^1$ and $\M(\Lambda_I)\simeq k_F^{\times}\times k_F^1$.  Furthermore, $\M(\Lambda_I)$ is a maximal torus in $\M(\Lambda_x)$ (resp. $\M(\Lambda_y)$) and $\P(\Lambda_I)$ is equal to the preimage in $\P(\Lambda_x)$ (resp. $\P(\Lambda_y)$) of a Borel subgroup $B_x$ (resp. $B_y$), which we call \emph{standard}, under the projection map $\P(\Lambda_x)\rightarrow \M(\Lambda_x)$ (resp. $\P(\Lambda_y)\rightarrow \M(\Lambda_y)$). 

The \emph{affine Weyl group} $\widetilde{W}=N_G(T)/T^0$ of $\U(2,1)(F/F_0)$ is an infinite dihedral group generated by the cosets represented by the elements $w_x=\adiag(1,1,1)$ and $w_y=\adiag(\varpi_F,1,\varpi_F^{-1})$. Furthermore, we have $\P(\Lambda_x)=\P(\Lambda_I)\cup \P(\Lambda_I)w_x\P(\Lambda_I)$ and $\P(\Lambda_y)=\P(\Lambda_I)\cup \P(\Lambda_I)w_y\P(\Lambda_I)$.
\subsection{Reduction modulo-$\ell$}
 Let $\Ql$ be an algebraic closure of the $\ell$-adic numbers, $\Zl$ be the ring of integers of $\Ql$, $\Gamma$ be the unique maximal ideal of $\Zl$, and $\Fl=\Zl/\Gamma$ be the residue field which is an algebraic closure of the finite field with $\ell$-elements.  Let $\mathfrak{Gr}_R(G)$ denote the \emph{Grothendieck group} of $R$-representations, i.e. the free abelian group with $\mathbb{Z}$-basis $\Irr_R(G)$.  A representation in $ \mathfrak{R}_{\Ql}(G)$ will be called \emph{$\ell$-adic} and a representation in $\mathfrak{R}_{\Fl}(G)$ will be called \emph{$\ell$-modular}.  We say $\ell$ is \emph{banal} for $G$ if it does not divide the pro-order of any compact open subgroup of $G$.  

Let $(\pi,\mathcal{V})$ be a finite length $\ell$-adic representation of $G$.  We call $\pi$ \emph{integral} if $\pi$ stabilises a $\Zl$-lattice $\mathcal{L}$ in $\mathcal{V}$.  In this case $\pi$ stabilises $\Gamma\mathcal{L}$ and $\pi$ induces a finite length $\ell$-modular representation on the space $\mathcal{L}/\Gamma\mathcal{L}$.  In general, this depends on the choice of the lattice $\mathcal{L}$.  However, due to Vign\'eras \cite[Theorem 1]{Vigint}, the semisimplification of $\mathcal{L}/\Gamma\mathcal{L}$ is independent of the lattice chosen and we define $r_{\ell}(\pi)$, the \emph{reduction modulo-$\ell$} of $\pi$, to be this semisimple $\ell$-modular representation.  If $\pi$ is a finite length $R$-representation of $G$ we write $[\pi]$ for the semisimplification of $\pi$ in $\mathfrak{Gr}_R(G)$.  

We fix choices of square roots of $p$ in $\Ql^\times$ and $\Fl^\times$ such that our chosen square root of $p$ in $\Fl^\times$ is the reduction modulo-$\ell$ of our chosen square root of $p$ in $\Ql^\times$ and make use of these choices in our definitions of normalised parabolic induction and the Jacquet module.

Parabolic induction preserves integrality and commutes with reduction modulo-$\ell$; if $P=M\ltimes N$ is a parabolic subgroup of $G$ and $\sigma$ is a finite length integral $\ell$-adic representation of $M$ then $r_{\ell}(i_P^G(\sigma))\simeq \left[i_P^G(r_{\ell}(\sigma))\right].$ Furthermore, compact induction commutes with reduction modulo-$\ell$; if $H$ is a closed subgroup of $G$, $\sigma$ an integral finite length representation of $H$, such that $\ind_H^G(\sigma)$ is finite length, then $r_{\ell}(\ind_H^G(\sigma))=[\ind_H^G(r_{\ell}(\sigma))]$.  For classical groups, due to Dat \cite{Dat}, the Jacquet module preserves integrality and commutes with reduction modulo-$\ell$; if $P=M\ltimes N$ is a parabolic subgroup of $G$ and $\pi$ is a finite length integral $\ell$-adic representation of $G$ then $
r_{\ell}(r^G_P(\pi))\simeq \left[r^G_P(r_{\ell}(\pi))\right].$ This implies that the reduction modulo-$\ell$ of a finite length cuspidal integral $\ell$-adic representation is cuspidal.  

An irreducible $R$-representation is admissible, due to Vign\'eras \cite[II 2.8]{Vig96}.  If $\pi$ is an $R$-representation we let $\widetilde{\pi}$ or $\pi^{\sim}$ denote the contragredient representation of $\pi$.

The abelian category $\mathfrak{R}_R(G)$ has a decomposition as a direct product of full subcategories $\mathfrak{R}_R^x(G)$, consisting of all representations all of whose irreducible subquotients have level $x$ for $x\in\mathbb{Q}_{\geqslant 0}$, which is preserved by parabolic induction and the Jacquet functor, by Vign\'eras \cite[II 5.8 \& 5.12]{Vig96}.

\section{Cuspidal representations of $\U(1,1)(k_F/k_0)$ and $\U(2,1)(k_F/k_0)$}\label{cuspU11U21}

 Our description of the supercuspidal $\ell$-adic representations of $\U(1,1)(k_F/k_0)$ and $\U(2,1)(k_F/k_0)$ and the decomposition of the $\ell$-adic principal series follow from similar arguments made for $\GL_2(k_F)$ and $\SL_2(k_F)$ by Digne--Michel \cite[\S 15.9]{DM91}.  The character tables of both groups were first computed by Ennola \cite{ennola} and the $\ell$-modular representations of $\U(2,1)(k_F/k_0)$ were first studied by Geck \cite{G90}.  In this section, let $H=\U(1,1)(k_F/k_0)$ and $G=\U(2,1)(k_F/k_0)$.  We call a torus minisotropic if it is equal to the points of a minisotropic torus of the corresponding algebraic group.
 
\subsection{Cuspidals of $\U(1,1)(k_F/k_0)$} 

\subsubsection{Cuspidals}
There are $\frac{q^2+q}{2}$ irreducible $\ell$-adic supercuspidal representations of $H$.  These can be parametrised by the regular irreducible characters of the minisotropic tori of $H$. There is only one conjugacy class of minisotropic tori in $G$, which is isomorphic to $k_F^1\times k_F^1$, hence a character of this torus corresponds to two characters of $k_F^1$.  Furthermore, this character is regular if and only if it corresponds to two distinct characters of $k_F^1$.  Thus the $\ell$-adic supercuspidals can be parametrised by unordered pairs of distinct irreducible characters of $k_F^1$.  Let $\chi_1,\chi_2$ be distinct $\ell$-adic characters of $k_F^1$. Let $\sigma(\chi_1,\chi_2)$ denote the $\ell$-adic supercuspidal representation parametrised by the set $\{\chi_1,\chi_2\}$. 

Using Clifford Theory the decomposition numbers for $H$ follow from the well known decomposition numbers of $\SU(1,1)(k_F/k_0)\simeq \SL_2(k_0)$.  We have $|H|=q(q-1)(q+1)$ hence, because $q$ is odd, there are four cases to consider: $\ell\mid q-1$, $\ell\mid q+1$, $\ell=2$, and $\ell$ is prime to $(q^2-1)$. 

All irreducible $\ell$-modular cuspidal representations of $H$ are isomorphic to the reduction modulo-$\ell$ of an irreducible $\ell$-adic supercuspidal representation. If $\chi$ is an $\ell$-adic character we let $\overline{\chi}$ denote its reduction modulo-$\ell$.  If $\chi_1',\chi_2'$ are $\ell$-adic characters of $k_F^1$, we have $r_{\ell}(\sigma(\chi_1,\chi_2))=r_{\ell}(\sigma(\chi_1',\chi_2'))$ if and only if $\{\overline{\chi}_1,\overline{\chi}_2\}=\{\overline{\chi}_1',\overline{\chi}'_2\}$.  We let $\overline{\sigma}(\overline{\chi}_1,\overline{\chi}_2)=r_{\ell}(\sigma(\chi_1,\chi_2))$. Furthermore, $\overline{\sigma}(\overline{\chi}_1,\overline{\chi}_2)$ is supercuspidal if and only if $|\{\overline{\chi}_1,\overline{\chi}_2\}|=2$ and we have $\overline{\sigma}(\overline{\chi}_1,\overline{\chi}_2)=\overline{\sigma}(\overline{\chi}_2,\overline{\chi_1})$.  Hence the irreducible cuspidal non-supercuspidal $\ell$-modular representations of $H$ are parametrised by the $\ell$-modular characters of $k_F^1$ and, if $\overline{\chi}$ is an $\ell$-modular character of $k_F^1$ equal to the reduction modulo-$\ell$ of two distinct $\ell$-adic characters of $k_F^1$, we let  $\overline{\sigma}(\overline{\chi})=\overline{\sigma}(\overline{\chi},\overline{\chi})$.  When $\ell\nmid q+1$ all irreducible cuspidal $\ell$-modular representations are supercuspidal.  

\subsubsection{Cuspidal non--supercuspidals when $\ell\mid q+1$}
Let $\ell^a\mid\mid q+1$, then there are $\frac{q+1}{\ell^a}$ cuspidal non--supercuspidal $\ell$-modular representations denoted by $\overline{\sigma}(\overline{\chi})$; these occur as the reduction modulo $\ell$ of $\sigma(\chi_1,\chi_2)$ when $\overline{\chi}=\overline{\chi}_1=\overline{\chi}_2$.  Let $T=\{\diag(x,\overline{x}^{-1}):x\in k_F\}$ be the maximal diagonal torus of $H$ and $B_H$ the standard Borel subgroup containing $T$.   The principal series representations  $i^H_{B_H}(\overline{\chi}\circ\xi)\simeq i^H_{B_H}(\overline{1})(\overline{\chi}\circ\det)$ are uniserial of length three with $(\overline{\chi}\circ\det)$ appearing as the unique irreducible subrepresentation and the unique irreducible quotient and unique irreducible cuspidal subquotient $\overline{\sigma}(\overline{\chi})$.

\subsection{Cuspidals of $\U(2,1)(k_F/k_0)$}
\subsubsection{$\ell$-adic supercuspidals}There are two conjugacy classes of minisotropic tori in $G$ which give rise to two classes of irreducible supercuspidal $\ell$-adic representations coming from regular irreducible characters of these tori.  Let $E$ be an unramified cubic extension of $F$.   One conjugacy class of the minisotropic tori has representatives isomorphic to $k_F^1\times  k_F^1\times k_F^1$; the other conjugacy class has representatives isomorphic to $k_E^1$.  However, in contrast to $H$, the irreducible representations parametrised by the irreducible regular characters of these tori do not constitute all the irreducible supercuspidal representations of $G$; additionally there exist unipotent supercuspidal representations of $G$.   Thus we have three classes of $\ell$-adic supercuspidals:
\begin{enumerate}
\item There are $\frac{(q+1)q(q-1)}{6}$ $\ell$-adic supercuspidals of dimension $(q-1)(q^2-q+1)$ parametrised by the irreducible regular characters of $k_F^1\times k_F^1\times k_F^1$.  An irreducible $\ell$-adic character of $k_F^1\times k_F^1\times k_F^1$ is of the form $\chi_{1}\otimes\chi_2\otimes\chi_3$, with $\chi_1,\chi_2,\chi_3$ irreducible $\ell$-adic characters of $k_F^1$, and is regular if and only if $|\{\chi_1,\chi_2,\chi_3\}|=3$.  We let $\sigma(\chi_1,\chi_2,\chi_3)$ denote the $\ell$-adic supercuspidal corresponding to the set $\{\chi_1,\chi_2,\chi_3\}$.
\item There are $\frac{(q+1)q(q-1)}{3}$ $\ell$-adic supercuspidals of dimension $(q-1)(q+1)^2$ parametrised by the irreducible regular characters of $k_E^1$.  An irreducible $\ell$-adic character $\psi$ of $k_E^1$ is regular if and only if $\psi^{q+1}\neq 1$.  We let $\tau(\psi)$ denote the $\ell$-adic supercuspidal representation corresponding to $\psi$.
\item There are $(q+1)$ unipotent $\ell$-adic supercuspidals of dimension $q(q-1)$.  These can be parametrised by the irreducible characters of $k_F^1$.  We write $\nu(\chi)$ for the unipotent $\ell$-adic supercuspidal representation corresponding to the irreducible $\ell$-adic character $\chi$ of $k_F^1$.
\end{enumerate}

\subsubsection{$\ell$-modular cuspidals}\label{u3finitecusplmod}
We have $|G|=q^3(q - 1)(q + 1)^3(q^2 - q + 1)$ hence there are six cases to consider:  $\ell=2$, $\ell=3$ and $\ell\mid q+1$, $\ell\mid q-1$, $\ell\mid q+1$, $\ell\mid q^2-q+1$, and $\ell$ is prime to $(q - 1)(q + 1)(q^2 - q + 1)$. When $\ell\neq 2$, the decomposition numbers can be obtained from Geck \cite{G90} and Okuyama--Waki \cite{OWU} using Clifford theory. Parabolic induction of the trivial character is completely described in Hiss \cite[Theorem $4.1$]{HissU3}. When $\ell\mid q-1$ or $\ell\mid q+1$, all irreducible cuspidal $\ell$-modular representations lift to irreducible cuspidal $\ell$-adic representations.  Analogously to the two dimensional case, we write $\overline{\nu}(\overline{\chi})=r_{\ell}(\nu(\chi))$, $\overline{\tau}(\overline{\psi})=r_{\ell}(\tau(\psi))$ and $\overline{\sigma}(\overline{\chi}_1,\overline{\chi}_2,\overline{\chi}_3)=r_{\ell}(\sigma(\chi_1,\chi_2,\chi_3))$.  

When $\ell\neq 3$ and $\ell\mid q^2-q+1$ we have irreducible $\ell$-modular cuspidal representations which do not lift: if $\psi$ is an $\ell$-adic character of $k_E^1$ such that $\psi^{q+1}\neq 1$, but $\overline{\psi}^{q+1}=\overline{1}$ then $r_{\ell}(\tau(\psi))=\overline{\nu}(\overline{\chi})\oplus\overline{\tau}^+(\overline{\chi})$, where $\overline{\chi}$ is the character of $k_F^1$ such that $\overline{\psi}=\overline{\chi}\circ\xi$ where $\xi(x)=x^{q-1}$, and $\overline{\tau}^+(\overline{\chi})$ does not lift.   When $\ell=2$ and $4\mid q-1$ we also have cuspidal representations which do not lift: if $\psi$ is an $\ell$-adic character of $k_E^1$ such that $\psi^{q+1}\neq 1$, but $\overline{\psi}^{q+1}=\overline{1}$ then $r_{\ell}(\tau(\psi))=\overline{\nu}(\overline{\chi})\oplus\overline{\nu}(\overline{\chi})\oplus\overline{\tau}^+(\overline{\chi})$, where $\overline{\chi}$ is the character of $k_F^1$ such that $\overline{\psi}=\overline{\chi}\circ\xi$ where $\xi(x)=x^{q-1}$, and $\overline{\tau}^+(\overline{\chi})$ does not lift.  All other irreducible cuspidal $\ell$-modular representations of $G$ lift to $\ell$-adic representations and we use the same notation as before. 

\subsubsection{$\ell$-adic principal series}

Let $T=\{\diag(x,y,\overline{x}^{-1}):x\in k_F,~y\in k_F^1\}$ be the maximal diagonal torus in $G$, and $B$ be the standard Borel subgroup of $G$ containing $T$.

Let $\chi_1$ be an $\ell$-adic character of $k_F^\times$ and $\chi_2$ an $\ell$-adic character of $k_F^1$.  Let $\chi$ be the irreducible character of $T$ defined by $\chi(\diag(x,y,x^{-q}))=\chi_1(x)\chi_2(xyx^{-q})$.   The character $\chi$ is regular if and only if $\chi_1^{q+1}\neq 1$ and in this case the principal series representation ${i}_{B}^{G}(\chi)$ is irreducible.  

If $\chi_1^{q+1}=1$ then $\chi_1=\chi'_1\circ\xi$ where $\xi(x)=x^{q-1}$ and $\chi'_1$ is an $\ell$-adic character of $k_F^1$.  If $\chi'_1=1$, or equivalently $\chi_1=1$, then
\begin{align*}
{i}_{B}^{G}(\chi)=1_{G}(\chi_2\circ\det)\oplus \St_G(\chi_2\circ\det)
\end{align*}
where $\St_G$ is an irreducible $q^3$-dimensional representation of $G$.  If $\chi'_1\neq 1$ then
\begin{align*}
{i}_{B}^{G}(\chi)=R_{1_H(\chi'_1)}(\chi_2\circ\det)\oplus R_{\St_H(\chi'_1)}(\chi_2\circ\det)
\end{align*}
where $R_{1_H(\chi'_1)}$ is an irreducible $(q^2-q+1)$-dimensional representation of $G$ and $R_{\St_H(\chi'_1)}$ is an irreducible $(q(q^2-q+1))$-dimensional representation of $G$.  The reducibility here comes from inducing first to the Levi subgroup $L^*=\U(1,1)(k_F/k_0)\times\U(1)(k_F/k_0)$ which is not contained in any proper rational parabolic subgroup of $G$.  Here $1_H$ and $\St_H$ denote the trivial and Steinberg representations of $\U(1,1)(k_F/k_0)$ and $R$ a generalised induction from $L^*$ to $G$.

\subsubsection{Cuspidal subquotients of $\ell$-modular principal series}

If  $\ell\neq 2$ and $\ell\mid q-1$ or $\ell$ is prime to $(q-1)(q+1)(q^2-q+1)$ then all irreducible cuspidal $\ell$-modular representations are supercuspidal and the principal series representations are all semisimple.

 Let $\overline{\chi}_2$ be an $\ell$-modular character of $k_F^1$.  We first describe the $\ell$-modular principal series representations ${i}_{B}^{G}(\overline{1})(\overline{\chi}_2\circ\det)$ in all the cases where cuspidal subquotients appear.
\begin{enumerate}
\item If $\ell\neq 3$ and $\ell\mid q^2-q+1$,  ${i}_{B}^{G}(\overline{1})(\overline{\chi}_2\circ\det)$ are uniserial of length three with $(\overline{\chi}\circ\det)$ appearing as the unique irreducible subrepresentation and the unique irreducible quotient and $\overline{\tau}^+(\overline{\chi})$ as the unique irreducible cuspidal subquotient.  
\item If $\ell\neq 2$ and $\ell\mid q+1$ or $\ell=2$ and $4\mid q+1$ then ${i}_{B}^{G}(\overline{1})(\overline{\chi}\circ\det)$ have irreducible cuspidal subquotients  $\overline{\nu}(\overline{\chi})$ and $\overline{\sigma}(\overline{\chi})=\overline{\sigma}(\overline{\chi},\overline{\chi},\overline{\chi})$.   The principal series representations ${i}_{B}^{G}(\overline{1})(\overline{\chi}\circ\det)$ are uniserial of length five with $(\overline{\chi}\circ\det)$ appearing as the unique irreducible subrepresentation and the unique irreducible quotient.  A maximal cuspidal subquotient of ${i}_{B}^{G}(\overline{1})(\overline{\chi}\circ\det)$ is uniserial of length three with $\overline{\nu}(\overline{\chi})$ appearing as the unique irreducible quotient and the unique irreducible subrepresentation, and remaining subquotient $\overline{\sigma}(\overline{\chi})$.

\item If $\ell=2$ and $4\mid q-1$ then ${i}_{B}^{G}(\overline{1})(\overline{\chi}\circ\det)$ has length four with $(\overline{\chi}\circ\det)$ appearing as the unique irreducible subrepresentation and the unique irreducible quotient, and cuspidal subquotient $\overline{\nu}(\overline{\chi})\oplus\overline{\tau}^{+}(\overline{\chi})$.  
\end{enumerate}

Now let $\overline{\chi}_1'$ and $\overline{\chi}_2$ be $\ell$-modular characters of $k_F^1$ with $\overline{\chi}_1'$ non-trivial and let $\overline{\chi}_1=\overline{\chi}_1'\circ \xi$.  Let $\overline{\chi}$ be the $\ell$-modular character of $T$ defined by $\overline{\chi}(\diag(x,y,x^{-q}))=\overline{\chi}_1(x)\overline{\chi}_2(xyx^{-q})$.  If $\ell\nmid q+1$ then $i^G_B(\overline{\chi})$ does not possess any cuspidal subquotients.
If $\ell\mid q+1$ then ${i}_{B}^{G}(\overline{\chi})$ is uniserial of length three with $ \overline{R}_{\overline{1}_H(\overline{\chi}'_1)}(\overline{\chi}_2\circ\det)$ appearing as the unique irreducible subrepresentation and the unique irreducible quotient and cuspidal subquotient $\sigma(\overline{\chi}'_1,\overline{\chi}'_1,\overline{\chi}_2)$.  This follows from Bonnaf\'e-Rouquier \cite[Theorem 11.8]{BR03} and the principal block of $H$ as $\chi$ corresponds to a semisimple element with centraliser $H\times k_F^1$ in the dual group. 

\section{Irreducible cuspidal $R$-representations of $\U(2,1)(F/F_0)$}\label{sect4}

\subsection{Types and Hecke algebras}
Let $G=\U(2,1)(F/F_0)$. We construct the irreducible cuspidal representations of $G$ by compact induction from the irreducible representations of its compact open subgroups.  We review some general theory first.  By an \emph{$R$-type}, we mean a pair $(K,\sigma)$ consisting of a compact open subgroup $K$ of $G$ and an irreducible $R$-representation $\sigma$ of $K$.  Given an $R$-type we consider the compactly induced representation $\ind_K^G(\sigma)$ of $G$; the goal being to find pairs $(K,\sigma)$ such that $\ind_K^G(\sigma)$ is irreducible and cuspidal. Let $\pi\in\Irr_R(G)$,  we say that $\pi$ \emph{contains} the $R$-type $(K,\sigma)$ if $\pi$ is a quotient of $\ind_K^G(\sigma)$. 

Let $(K,\sigma)$ be an $R$-type in $G$ and $\mathcal{W}$ be the space of $\sigma$.  The \emph{spherical Hecke algebra} $\mathcal{H}(G,\sigma)$ of $\sigma$ is the $R$-module consisting of the set of all functions $f:G\rightarrow \End_R(\mathcal{W})$ such that the support of $f$ is a finite union of double cosets in $K\backslash G/K$ and $f$ transforms by $\sigma$ on the left and the right, i.e. $f(k_1gk_2)=\sigma(k_1)f(g)\sigma(k_2)$, for all $k_1,k_2\in K$ and all $g\in G$. The product in $\mathcal{H}(G,\sigma)$ is given by convolution; if $f_1, f_2\in\mathcal{H}(G,\sigma)$ then
\begin{align*}
f_1\star f_2(h)=\sum_{G/K}f_1(g)f_2(g^{-1}h).
\end{align*}
The spherical Hecke algebra $\mathcal{H}(G,\sigma)$ is isomorphic to $\End_G(\ind_K^G(\sigma))$ where multiplication in $\End_G(\ind_K^G(\sigma))$ is defined by composition.  For $g\in G$, let $I_g(\sigma)=\Hom_K(\sigma,\ind_{K\cap K^g}^K\sigma^g)$ and let $I_G(\sigma)=\{g\in G:I_g(\sigma)\neq 0\}.$

Let $\mathcal{M}(G,\sigma)$ denote the category of right $\mathcal{H}(G,\sigma)$-modules. Define $M_{\sigma}:\mathfrak{R}_R(G)\rightarrow \mathcal{M}(G,\sigma)$ by $\pi\mapsto \Hom_G(\ind_K^G(\sigma),\pi)$ which is a (right) $\End_G(\ind_K^G(\sigma))$-module by pre-composition.  In the $\ell$-adic case, if $(K,\sigma)$ is a type in the sense of Bushnell--Kutzko \cite[Page $584$]{BK98}, $M_{\sigma}$ induces an equivalence of categories between the full subcategory of $\mathfrak{R}_R(G)$ of representations all of whose irreducible subquotients contain $(K,\sigma)$ and $\mathcal{M}(G,\sigma)$.

An $R$-representation $(\pi,\mathcal{V})$ of $G$ is called \emph{quasi-projective} if for all $R$-representations $(\sigma,\mathcal{W})$ of $G$, all surjective $\Phi\in\Hom_G(\mathcal{V},\mathcal{W})$ and all $\Psi\in\Hom_G(\mathcal{V},\mathcal{W})$ there exists $\Xi\in\End_G(\mathcal{V})$ such that $\Psi=\Phi\circ\Xi$.

\begin{thm}[{Vign\'eras \cite[Appendix, Theorem $10$]{vignerasselecta}}]\label{qplemma}
Let $\pi$ be a quasi-projective finitely generated $R$-representation of $G$.  The map $\rho\mapsto \Hom_G(\pi,\rho)$ induces a bijection between the irreducible quotients of $\pi$ and the simple right $\End_G(\pi)$-modules.
\end{thm}

Let $P$ be a parabolic subgroup of $G$ with Levi decomposition $P=M\ltimes N$.  Let $P^{op}$ be the opposite parabolic subgroup of $P$ with Levi decomposition $P^{op}=M\ltimes N^{op}$.  Let $K^+=K\cap N$ and $K^{-}=K\cap N^{op}$.  An element $z$ of the centre of $M$ is called \emph{strongly $(P,K)$ positive} if \begin{enumerate}
\item $zK^+z^{-1}\subset K^+$ and $zK^{-}z^{-1}\supset K^{-}$;
\item for all compact subgroups $H_1,H_2$ of $N$ (resp. $N^{op}$), there exists a positive (resp. negative) integer $m$ such that $z^mH_1z^{-m}\subset H_2$.
\end{enumerate}
Let $(K_M,\sigma_M)$ be an $R$-type of $M$.  An $R$-type $(K,\sigma)$ is called a \emph{$G$-cover} of $(K_M,\sigma_M)$ relative to $P$ if we have:
\begin{enumerate}
\item $K\cap M=K_M$ and we have an Iwahori decomposition $K=K^{-}K_M K^{+}$.
\item $\Res^K_{K_M}(\sigma)=\sigma_M$, $\Res^K_{K^+}(\sigma)$ and $\Res^K_{K^{-}}(\sigma)$ are both multiples of the trivial representation.
\item There exists a strongly $(P,K)$ positive element $z$ of the centre of $M$ such that the double coset $Kz^{-1}K$ supports an invertible element of $\mathcal{H}_R(G,\sigma)$.
\end{enumerate}

The point being that the properties of a $G$-cover allow one to define an injective homomorphism of algebras $j_P:\mathcal{H}(M,\sigma_M)\rightarrow \mathcal{H}(G,\sigma)$ and hence a (normalised) restriction functor $(j_P)^*:\mathcal{M}(G,\sigma)\rightarrow \mathcal{M}(M,\sigma_M)$, see Bushnell--Kutzko \cite[Page $585$]{BK98} and Vign\'eras \cite[II \S 10]{vignerasselecta}.   

\begin{thm}[{Vign\'eras \cite[II \S 10.1]{vignerasselecta}}]\label{jacquetcommutes}
Let $\pi$ be a finitely generated $\ell$-modular representation of $G$.  We have an isomorphism $(j_P)^*(M_\sigma(\pi))\simeq M_{\sigma_M}(r^G_P(\pi))$ of representations of $M$.
\end{thm}

\subsection{Level zero $\ell$-modular representations}  
An irreducible representation $\pi$ of $G$ is of level zero if it has nontrivial invariants under the pro-$p$ unipotent radical of some maximal parahoric subgroup of $G$.

Let $\Lambda$ be a self--dual $\mathfrak{o}_F$-lattice sequence in $V$ and $\P(\Lambda)$ the associated parahoric subgroup in $G$.  We define \emph{parahoric induction} $\I_\Lambda:\mathfrak{R}_R(\M(\Lambda))\rightarrow \mathfrak{R}_R(G)$ by
\begin{align*}
\I_\Lambda(\sigma)=\ind_{\P(\Lambda)}^G(\widetilde{\sigma}),
\end{align*}
 for $\sigma$ an $R$-representation of $\M(\Lambda)$ and $\widetilde{\sigma}$ the inflation of $\sigma$ to $\P(\Lambda)$ defining $\P_1(\Lambda)$ to act trivially. This functor has a right adjoint, \emph{parahoric restriction} $\R_\Lambda:\mathfrak{R}_R(G)\rightarrow\mathfrak{R}_R(\M(\Lambda))$, defined by
\begin{align*}
\R_\Lambda(\pi)=\pi^{\P_1(\Lambda)}, 
\end{align*}
for $\pi$ an $R$-representation of $G$. Parahoric induction and restriction are exact functors.

We have the following important lemma due to Vign\'eras \cite{vigneras}.  In \textit{op.\,cit.} the statement is for a general $p$-adic reductive group $G$.
\begin{lemma}[Vign\'eras \cite{vigneras}]\label{lemma}
Let $\P_1=\P(\Lambda_1)$ and $\P_2=\P(\Lambda_2)$ be parahoric subgroups of $G$.  Let $\sigma$ be a finite length representation of $\M(\Lambda_2)$ and fix a set $D_{1,2}$ of distinguished double coset representatives of $\P_1\backslash G/\P_2$.  We have an isomorphism
\begin{align*}
\R_{\Lambda_1}\circ\I_{\Lambda_2}(\sigma)\simeq\bigoplus_{n\in D_{1,2}}{i}_{ P_{\Lambda_1,n\Lambda_2}}^{\M(\Lambda_1)}\left( r^{\M(\Lambda_2)}_{P_{\Lambda_2,n^{-1}\Lambda_1}}(\sigma)\right)^n.
\end{align*}
\end{lemma}

\begin{lemma}\label{lemma111}
Let $\P(\Lambda_1)$ and $\P(\Lambda_2)$ be parahoric subgroups of $G$ associated to the $\mathfrak{o}_F$-lattice sequences $\Lambda_1$ and $\Lambda_2$ in $V$.  Suppose that $\P(\Lambda_1)$ is maximal and let $\sigma$ be an irreducible cuspidal representation of $\M(\Lambda_1)$.  We have
\begin{align*}
\R_{\Lambda_2}\circ\,\I_{\Lambda_1}(\sigma)=\begin{cases}\sigma&\text{if }\P(\Lambda_2)\text{ is conjugate to }\P(\Lambda_1)\text{ in }G;\\0&\text{otherwise.}\end{cases}
\end{align*}
\end{lemma}

\begin{proof}
 If $\P(\Lambda_2)$ is conjugate to $\P(\Lambda_1)$ then $\P(\Lambda_1)\backslash G/\P(\Lambda_2)\simeq 1$ hence $\R_{\Lambda_2}\circ\I_{\Lambda_1}(\sigma)\simeq \sigma$ by Lemma \ref{lemma}.  If $\P(\Lambda_2)$ is not conjugate to $\P(\Lambda_1)$ then for all $n\in D_{1,2}$ the parabolic subgroup $P_{\Lambda_2,n^{-1}\Lambda_1}$ is a proper parabolic subgroup of $\M(\Lambda_2)$.  Hence $\R_{\Lambda_2}\circ\I_{\Lambda_1}(\sigma)\simeq 0$ by Lemma \ref{lemma} and cuspidality of $\sigma$.
\end{proof}

\subsection{Positive level cuspidal $\ell$-modular representations} 
\subsubsection{Semisimple strata and characters}\label{431}
Let $[\Lambda,n,r,\beta]$ be a \emph{skew semisimple stratum} in $A$, see Stevens \cite[Definition $2.8$]{St08}.  Associated to $[\Lambda,n,r,\beta]$ and a fixed level one character of $F_0^\times$ are: \begin{enumerate}
\item A decomposition $V=\bigoplus_{i=1}^l V_i$, orthogonal with respect to $h$, and a sum of field extensions $E=\bigoplus_{i=1}^l E_i$ of $E$ such that $\Lambda=\bigoplus_{i=1}^l \Lambda_i$ with $\Lambda_i$ an $\mathfrak{o}_{E_i}$-lattice sequence in $V_i$ (we say that $\Lambda$ is an \emph{$\mathfrak{o}_E$-lattice sequence} and write $\Lambda_E$ when we are considering $\Lambda$ as such). 
\item The $F_0$-points of a product of unramified unitary groups defined over $F_0$,  $\GE=\prod_{i=1}^l \GEi$.
\item  Compact open subgroups $H(\Lambda,\beta)\subseteq J(\Lambda,\beta)$ of $G$ with decreasing filtrations by pro-$p$ normal compact open subgroups $H^n(\Lambda,\beta)=H(\Lambda,\beta)\cap \P_n(\Lambda)$ and $J^n(\Lambda,\beta)=J(\Lambda,\beta)\cap \P_n(\Lambda)$, $n\geqslant 1$.  When $\Lambda$ is fixed we write $J=J(\Lambda,\beta)$, $H=H(\Lambda,\beta)$, and use similar notation for their filtration subgroups. We have $J=P(\Lambda_E) J^1$ where $\P(\Lambda_E)$ is the parahoric subgroup of $\GE$ obtained by considering $\Lambda$ as an $\mathfrak{o}_E$-lattice sequence.   
\item A set of \emph{semisimple characters} $\mathcal{C}_{-}(\Lambda,r,\beta)$ of $H^1(\Lambda,\beta)$.   We let $\mathcal{C}_{-}(\Lambda,\beta)=\mathcal{C}_{-}(\Lambda,0,\beta)$. \end{enumerate}
Let $[\Lambda_i,n,0,\beta]$, for $i=1,2$, be skew semisimple strata in $A$.  For all $\theta_1\in\mathcal{C}_{-}(\Lambda_1,\beta)$, there is a unique $\theta_2\in\mathcal{C}_{-}(\Lambda_2,\beta)$ such that $1\in I_G(\theta_1,\theta_2)$ by Stevens \cite[Proposition 3.32]{St05}.  This defines a bijection 
\begin{align*}\tau_{\Lambda_1,\Lambda_2,\beta}: \mathcal{C}_{-}(\Lambda_1,\beta)\rightarrow\mathcal{C}_{-}(\Lambda_2,\beta)\end{align*}
and we call $\theta_2=\tau_{\Lambda_1,\Lambda_2,\beta}(\theta_1)$ the \emph{transfer} of $\theta_1$.

The skew semisimple strata in $A$ fall into three classes.
\begin{enumerate}
\item Skew simple strata $[\Lambda,n,0,\beta]$; where $E$ is a field. \begin{enumerate}\item  If $E=F$ we say that $[\Lambda,n,0,\beta]$ is a scalar skew simple stratum.  Then $J/J^1= \P(\Lambda)/\P_1(\Lambda)$ is isomorphic to $\GL_1(k_F)\times\U(1)(k_F/k_0)$, $\U(1,1)(k_F/k_0)\times\U(1)(k_F/k_0)$ or $\U(2,1)(k_F/k_0)$.  
 \item Otherwise $E/F$ is cubic and $J/J^1\simeq \P(\Lambda_E)/\P_1(\Lambda_E)\simeq \U(1)(k_E/k_{E^0})$, a finite unitary group of order $q_{E^0}+1$ where 
\begin{align*}
q_{E^0}=\begin{cases}q_0^3&\text{if $E/F$ is unramified};\\ q_0&\text{if $E/F$ is ramified.}\end{cases}\end{align*} \end{enumerate}
\item Skew semisimple strata $[\Lambda,n,0,\beta]=[\Lambda_1,n_1,0,\beta_1]\oplus[\Lambda_2,n_2,0,\beta_2]$, not equivalent to a skew simple stratum, with $[\Lambda_i,n_i,0,\beta_i]$ skew simple strata in $\End_{F_0}(V_i)$, $i=1,2$.  Without loss of generality suppose that $V_1$ is one dimensional and $V_2$ is two dimensional. We have $J/J^1\simeq \prod_{i=1}^2 \P(\Lambda_{i,E})/\P_1(\Lambda_{i,E})$.  If $\beta_2\in F$ and $V_2$ is hyperbolic then $G_E\simeq U(1,1)(F/F_0)\times\U(1)(F/F_0)$ and $\P(\Lambda_{2,E})$ is a parahoric subgroup of $\U(1,1)(F/F_0)$ and need not be maximal.  If $\beta_2\in F$ and $V_2$ is anisotropic then $G_E\simeq \U(2)(F/F_0) \times \U(1)(F/F_0)$ is compact.  If $E_2/F$ is quadratic then it is ramified because there is a unique unramified extension of $F_0$ in each degree and $E_2^0/F_0$ is quadratic and also fixed by the involution.  Thus if $E_2/F$ is quadratic then $J/J^1\simeq \U(1)(k_F/k_0)\times\U(1)(k_F/k_0)$.
\item Skew semisimple strata $[\Lambda,n,0,\beta]=\bigoplus_{i=1}^3[\Lambda_i,n_i,0,\beta_i]$, not equivalent to a skew semisimple stratum of the first two classes, with $[\Lambda_i,n_i,0,\beta_i]$ skew simple strata in $\End_{F_0}(V_i)$.  In this case, $J/J^1\simeq \U(1)(k_F/k_0)\times\U(1)(k_F/k_0)\times\U(1)(k_F/k_0)$.
\end{enumerate}

We say that $\pi$ \emph{contains} the skew semisimple stratum $[\Lambda,n,0,\beta]$ if it contains a character $\theta\in\mathcal{C}_{-}(\Lambda,\beta)$.  

\begin{thm}[{Stevens \cite[Theorem $5.1$]{St05}}]\label{plss}
Let $\pi$ be an irreducible cuspidal $\ell$-modular representation of $G$.  Then $\pi$ contains a skew semisimple stratum $[\Lambda,n,0,\beta]$. 
\end{thm}

This is where we start in our construction of all irreducible cuspidal $\ell$-modular representations of $G$.

\subsubsection{Heisenberg representations}
Let $\theta\in\mathcal{C}_{-}(\Lambda,\beta)$.  By Stevens \cite[Corollary $3.29$]{St08}, there exists a unique irreducible representation $\eta$ of $J^1(\Lambda,\beta)$ which contains $\theta$.  We call such an $\eta$ a \emph{Heisenberg representation}.  Furthermore, by Stevens \cite[Proposition $3.31$]{St08}, \begin{equation*}
\dim_R(I_g(\eta))=\begin{cases} 1&\text{if }g\in J^1G_EJ^1,\\
0&\text{otherwise.}\end{cases}
\end{equation*}

\subsubsection{$\beta$-extensions}

Assume $P(\Lambda_E)$ is maximal.  A \emph{$\beta$-extension} of $\eta$ to $J=J(\Lambda,\beta)$ is an extension $\kappa$ with maximal intertwining, $I_G(\kappa)=I_G(\eta)$.  By Blasco \cite[Lemma $5.8$]{blasco}, for all maximal skew semisimple strata which are not skew scalar simple strata, $\beta$-extensions exist in the $\ell$-adic case for $G$ and for $\ell$-modular representations we obtain $\beta$-extensions by reduction modulo-$\ell$ from the $\ell$-adic extensions.  Let $[\Lambda,n,0,\beta]$ be a scalar skew simple stratum and $\theta\in \mathcal{C}_{-}(\Lambda,\beta)$.  Then $J^1=H^1=\P_1(\Lambda)$ and $J=\P(\Lambda)$ and $\theta=\chi\circ\det$ for some character $\chi$ of $\P_1(\Lambda)$ (\emph{c.f.} Bushnell--Kutzko \cite[Definition $3.23$]{BK93}).  The character $\chi$ extends to a character $\widetilde{\chi}$ of $F^1$ and we define $\kappa:J\rightarrow R^\times$ by $\kappa=\widetilde{\chi}\circ\det$.  Then $\kappa$ extends $\theta$ and is intertwined by all of $G$, hence is a $\beta$-extension.  Hence, in the maximal case, $\beta$-extensions exist.

Let $[\Lambda,n,0,\beta]$ be a skew semisimple stratum.  Suppose $P(\Lambda_E)$ is not maximal and choose a maximal parahoric subgroup $P(\Lambda^{m}_E)$ of $\GE$ associated to the $\mathfrak{o}_E$-lattice sequence $\Lambda^m_E$ in $V$ such that $P(\Lambda_E)\subset P(\Lambda^{m}_E)$.  This implies that $P(\Lambda)\subset P(\Lambda^{m})$.  Note that this is the case for unramified $\U(2,1)(E/F)$, but not for classical groups in general.  Let $\theta\in \mathcal{C}_{-}(\Lambda,\beta)$ and $\eta$ be the irreducible representation of $J^1_m=J^1(\beta,\Lambda)$ which contains $\theta$.  Let $\theta_{m}=\tau_{\Lambda,\Lambda^{m},\beta}(\theta)$ and $\eta_{m}$ be the irreducible representation of $J^1(\beta,\Lambda^{m})$ which contains $\theta_{m}$.  Let $\kappa_{m}$ be a $\beta$-extension of $\eta_{m}$.  

\begin{lemma}\label{compatbeta}
There exists a unique extension $\kappa$ of $\eta$ to $J$ such that $\Res^{J_m}_{P(\Lambda_E)J^1_m}(\kappa_{m})$ and $\kappa$ induce equivalent irreducible representations of $P(\Lambda_E)P_1(\Lambda)$.  
\end{lemma}

\begin{proof}
If $\P(\Lambda_E)$ is maximal then $\kappa_{m}=\kappa$ and there is nothing to prove.  In the $\ell$-adic case, by Stevens \cite[Lemma $4.3$]{St08}, there exists the required irreducible representation $\widetilde{\kappa}$ of $J$.  By reduction modulo-$\ell$, we have an irreducible $\ell$-modular representation $\kappa=r_{\ell}(\widetilde{\kappa})$ which extends $\eta$ such that
\[\left[\ind_{J}^{\P(\Lambda_E)\P_1(\Lambda)}\kappa\right]=\left[\ind^{\P(\Lambda_E)\P_1(\Lambda)}_{\P(\Lambda_E)J^1_m}\Res^{J_m}_{\P(\Lambda_E)J^1_m}(\kappa_{m})\right].\]
Furthermore, $J\subseteq I_{\P(\Lambda_E)\P_1(\Lambda)}(\kappa)\subseteq I_{\P(\Lambda_E)\P_1(\Lambda)}(\eta)=J$ hence by Mackey Theory $\ind_{J}^{\P(\Lambda_E)\P^1(\Lambda)}\kappa$ is irreducible.
\end{proof}

A \emph{$\beta$-extension} of $\eta$ is an extension $\kappa$ of $\eta$ to $J$ constructed in this way.  We call two $\beta$-extensions which induce equivalent representations, as in Lemma \ref{compatbeta}, \emph{compatible}.    With the next Lemma we show we can ``go backwards''  and from a $\beta$-extension defined in the minimal case we define two unique compatible $\beta$-extensions in the maximal case.  In this way we get a triple of \emph{compatible} $\beta$-extensions.  Let $\P(\Lambda_E^r)$ be a maximal parahoric subgroup of $\GE$ containing $\P(\Lambda_E)$ associated to the $\mathfrak{o}_E$-lattice sequence $\Lambda_E^r$ in $V$.  Let $\theta_{r}=\tau_{\Lambda,\Lambda^{r},\beta}(\theta)$, $\eta_{r}$ be the irreducible representation of $J^1(\beta,\Lambda^{r})$ which contains $\theta_{r}$, and $\kappa$ be a $\beta$-extension of $\eta$. 

\begin{lemma}\label{betascompat}
There exists a unique $\beta$-extension $\kappa_r$ of $\eta_r$ which is compatible with $\kappa$.
\end{lemma}

\begin{proof}
There exists a representation $\widehat{\kappa}$ of $P(\Lambda_E)J^1_r$ such that $\kappa$ and $\widehat{\kappa}$ induce equivalent representations of $\P(\Lambda_E)\P_1(\Lambda)$. Let $\kappa'$ be a $\beta$-extension of $\eta_r$.  The restriction of $\kappa'$ to $\P(\Lambda_E)J^1_r$ and $\widehat{\kappa}$ differ by a character $\chi$ of $B_r=\P(\Lambda_E)/\P^1(\Lambda_E^r)$ which is trivial on its unipotent part and intertwined by the non-trivial Weyl group element $w$.    By the Bruhat decomposition $\M_r=\M(\Lambda_E^r)=B_r\cup B_rwB_r$, hence $\chi$ is intertwined by the whole of $\M_r$ and extends to a character of $\M_r$.  Hence $\kappa_r=\kappa\otimes \chi^{-1}$ is a $\beta$-extension of $\eta_r$ which is compatible with $\kappa$. By reduction modulo-$\ell$, as in the proof of Lemma \ref{compatbeta}, we have the corresponding statement in the $\ell$-modular setting.
\end{proof}

\subsubsection{$\kappa$-induction and restriction}

Fix $[\Lambda,n,0,\beta]$ a skew semisimple stratum in $A$, $\theta\in\mathcal{C}_{-}(\Lambda,\beta)$, $\eta$ the unique Heisenberg representation containing $\theta$ and $\kappa$ a $\beta$-extension of $\eta$.

Let $\sigma$ be an $R$-representation of $\M(\Lambda_E)$ and $\widetilde{\sigma}$ denote the inflation of $\sigma$ to $J$ by defining $J^1$ to act trivially.  Define \emph{$\kappa$-induction} $\I_{\kappa}:\mathfrak{R}_R(\M(\Lambda_E))\rightarrow \mathfrak{R}_R(G)$ by
\begin{align*}
\I_{\kappa}(\sigma)=\ind_{J}^G(\kappa\otimes \widetilde{\sigma}).
\end{align*}
This functor has a right adjoint, \emph{$\kappa$-restriction} $\R_{\kappa}:\mathfrak{R}_R(G)\rightarrow\mathfrak{R}_R(\M(\Lambda_E))$, defined by
\begin{align*}
\R_{\kappa}(\pi)=\Hom_{J^1}(\kappa,\pi),
\end{align*}
for  $\pi$ an $R$-representation of $G$, where the action of $\M(\Lambda_E)$ is given by: let $f \in\Hom_{J^1}(\kappa,\pi)$, $m\in \M(\Lambda_E)$ and $j\in J$ represent the coset $m\in J/J^1$, $m\cdot f=\pi(j)\circ f\circ \kappa(j^{-1})$.

In the level zero case, we have $J=\P(\Lambda)$ and we can choose $\kappa$ to be trivial, thus we have $\I_{\kappa}=\I_{\Lambda}$ and $\R_{\kappa}=\R_{\Lambda}$.  Hence $\kappa$-restriction and induction generalise parahoric restriction and induction.   Related to $[\Lambda,n,0,\beta]$ we also have functors of parahoric induction $\I^E_{\Lambda}:\mathfrak{R}_R(M(\Lambda_E))\rightarrow\mathfrak{R}_R(\GE)$ and parahoric restriction $\R^E_{\Lambda}:\mathfrak{R}_R(\GE)\rightarrow\mathfrak{R}_R(M(\Lambda_E)) $ obtained by considering $\Lambda$ as an $\mathfrak{o}_E$-lattice sequence.

\begin{thm}[{Kurinczuk--Stevens \cite{RKSS}}]\label{pl=lz}
Let $[\Lambda^i,n,0,\beta]$, $i=1,2$, be skew semisimple strata.  Let $\theta_1\in\mathcal{C}_{-}(\Lambda^1,\beta)$ and $\theta_2=\tau_{\Lambda^1,\Lambda^2,\beta}(\theta_1)$.  For $i=1,2$, let $\eta_i$ be Heisenberg extensions of $\theta_i$, $\kappa_i$ be compatible $\beta$-extensions of $\eta_i$ and $\sigma$ an $R$-representation of $\M(\Lambda^1_E)$.  Then
\begin{align*}
\R_{\kappa_2}\circ\I_{\kappa_1}(\sigma)\simeq \R^E_{\Lambda^2}\circ \I^E_{\Lambda^1}(\sigma).
\end{align*}
\end{thm}

The proof of Theorem \ref{pl=lz} in {Kurinczuk--Stevens \cite{RKSS}} follows from a combination of Mackey theory, isomorphisms defined as in Bushnell--Kutzko \cite[Proposition $5.3.2$]{BK93}, and the computation of the intertwining spaces $\I_g(\eta_1,\eta_2)$ for $g\in G$ which are one-dimensional if $g\in \GE$ and zero otherwise.

\begin{lemma}\label{star}
Suppose we are in the setting of Lemma \ref{compatbeta} with compatible $\beta$-extensions $\kappa$ and $\kappa_{m}$.  Then, for all $\sigma\in\mathfrak{R}_R(\M(\Lambda_E))$, we have
\[\I_{\kappa}(\sigma)\simeq \ind_{J_{m}^1\P(\Lambda_E)}^G(\kappa_{m}\otimes \sigma)\]
and, for all $R$-representations $\pi$ of $G$, we have $\R_{\kappa}(\pi)\simeq\Hom_{J^1_{m}\P_1(\Lambda_E)}(\kappa_{m},\pi)$.
\end{lemma}

\begin{proof}
By transitivity of induction and Lemma \ref{compatbeta} we have $\I_{\kappa}(\sigma)\simeq \ind_{J_{m}^1\P(\Lambda_E)}^G(\kappa_{m}\otimes \sigma)$.  By reciprocity, for $\pi$ an $R$-representation of $G$, we have $\R_{\kappa}(\pi)\simeq\Hom_{J^1_{m}\P_1(\Lambda_E)}(\kappa_{m},\pi)$.
\end{proof}

Define \emph{$\widetilde{\kappa}$-induction} $\I_{\widetilde{\kappa}}:\mathfrak{R}_R(\M(\Lambda_E))\rightarrow\mathfrak{R}_R(G)$ by $\I_{\widetilde{\kappa}}(\sigma)=\ind_J^G(\widetilde{\kappa}\otimes\sigma)$ for $\sigma$ an $R$-representation of $\M(\Lambda_E)$.  This functor has a right adjoint, \emph{$\widetilde{\kappa}$-restriction} $\R_{\widetilde{\kappa}}:\mathfrak{R}_R(G)\rightarrow \mathfrak{R}_R(\M(\Lambda_E))$, defined by $\R_{\widetilde{\kappa}}(\pi)=\Hom_{J^1}(\widetilde{\kappa},\pi)$ where the action of $\M(\Lambda_E)$ on $\R_{\widetilde{\kappa}}(\pi)$ is defined analogously to $\kappa$-restriction.   In fact, $\widetilde{\kappa}$ is a $-\beta$-extension for the semisimple character $\theta^{-1}$ for the semisimple stratum $[\Lambda,n,0,-\beta]$.

\begin{lemma}\label{contracommutes}
Let $\pi$ be an $R$-representation of $G$ and $\sigma$ be an irreducible representation of $\M(\Lambda_E)$.  Then $\left(\R_{\kappa}(\pi)\right)^{\sim}\simeq \R_{\widetilde{\kappa}}(\widetilde{\pi})$ and if $\I_{\kappa}(\sigma)$ is irreducible then $\I_{\kappa}(\sigma)^{\sim}\simeq \I_{\widetilde{\kappa}}(\widetilde{\sigma})$.
\end{lemma}

\begin{proof}
We have an isomorphism of vector spaces $\Hom_{J^1}(\kappa,\pi)^{\sim}\simeq \Hom_{J^1}(\pi,\kappa)\simeq \Hom_{J^1}(\widetilde{\kappa},\widetilde{\pi})$ by Henniart--S\'echerre \cite[Proposition $2.6$]{HenniartSecherre} and checking the action of $J/J^1$ we have $\left(\R_{\kappa}(\pi)\right)^{\sim}\simeq \R_{\widetilde{\kappa}}(\widetilde{\pi})$.  If $\I_{\kappa}(\sigma)$ is irreducible then it is admissible and we have  $\I_{\kappa}(\sigma)^{\sim}\simeq \I_{\widetilde{\kappa}}(\widetilde{\sigma})$ by Vign\'eras \cite[I 8.4]{Vig96}.  
\end{proof}

Suppose $\P(\Lambda_E)$ is not maximal.  Let $\kappa_T=\Res^J_{T^0}(\kappa)$.   Define $\R_{\kappa_T,\Lambda}:\mathfrak{R}_R(T)\rightarrow\mathfrak{R}_R(\overline{T})$ by $\R_{\kappa_T,\Lambda}(\pi)=\Hom_{T^1}(\kappa_T,\pi)$.

\section{Exhaustion of cuspidal representations}

\subsection{Covers}
 Let $\lambda_T$ be an irreducible character of $T^0$. Suppose $J$ has Iwahori decomposition $J=(J\cap \overline{N})(J\cap T)(J\cap N)$ with respect to $B$ and define a character $\lambda$ of $J$ by $\lambda(j^{-}j_Tj^+)=\lambda_T(j_T)$, for $j^{-}\in (J\cap \overline{N})$, $j_T\in J\cap T$, and $j^{+}\in (J\cap N)$. 

\begin{lemma}\label{intchi}~
We have $w_x$ intertwines $\lambda$ if and only if $w_y$ intertwines $\lambda$.
\end{lemma}

\begin{proof}~
Suppose $w_x\in I_G(\lambda)$.  Then, as $w_x$ normalises $T^0$, $w_x$ normalises $\Res^J_{T^0}(\lambda)$.  For all $t\in T^0$ we have $w_x tw_x=w_ytw_y$, hence $w_y$ normalises $\Res^J_{T^0}(\lambda)$.  Let $j\in J\cap w_y Jw_y$ such that $j=w_yj'w_y$.  Using the Iwahori decomposition of $J$ we have $j=j_{\overline{N}} j_T j_N$ and $j'=j'_N j'_T j'_{\overline{N}}$ with $j_N, j'_N$ upper triangular unipotent, $j_{\overline{N}}, j'_{\overline{N}}$ lower triangular unipotent and $j_T,j'_T$ in $T$.  Thus \[j=w_y j'w_y^{-1}=(w_yj'_{N} w_y)(w_yj'_Tw_y)(w_yj'_{\overline{N}}w_y)\]
and, by uniqueness of the Iwahori decomposition, $j_{\overline{N}}=w_yj'_{N} w_y$, $j_T=w_yj'_Tw_y$ and $j_{N}=w_yj'_{\overline{N}} w_y$. Therefore $w_y\in I_G(\lambda)$.
\end{proof}

In the $\ell$-adic case our construction of covers is a special case of the general results of Stevens \cite[Propositions $7.10$ and $7.13$]{St08}. Let $[\Lambda,n,0,\beta]$ be a skew semisimple stratum in $A$ such that $\P(\Lambda_E)$ is not a maximal parahoric subgroup of $\GE$.  Let $\theta\in\mathcal{C}_{-}(\Lambda,\beta)$, $\eta$ be the unique Heisenberg representation containing $\theta$, $\kappa$ a $\beta$-extension of $\eta$ and $\sigma\in\Irr_R(J/J^1)$. Let $\kappa_T=\Res^J_{T^0}(\kappa)$ and set $\lambda_T=\kappa_T\otimes\sigma$.  Then $\lambda=\kappa\otimes \sigma$ by {\emph{op.\,cit.}}, i.e. $\kappa\otimes\sigma$ is trivial on the unipotent parts of $J$ (note that, in the notation of \emph{op.\,cit.}, in our case we have $J=J_B$).

\begin{lemma}\label{plcovers}
Let $\lambda_T=\kappa_T\otimes\sigma$.  Then $(J,\lambda)$ is a $G$-cover of $(T^0, \lambda_T)$. \end{lemma}

\begin{proof}
In the $\ell$-modular case, it remains to show that there exists a strongly $(B,J)$-positive element $z$ of the centre of $T$ such that $JzJ$ supports an invertible element of $\mathcal{H}(G,\lambda)$.  Let $\zeta=w_xw_y$.  For $g\in I_G(\lambda)$, because $\lambda$ is a character, $I_g(\lambda)\simeq R$ and there is a unique function in $f_g\in\mathcal{H}(G,J,\lambda)$ with support $J gJ$ such that $f_g(1)=1$.  We have $\zeta,\zeta^{-1}\in I_G(\lambda)$ hence $f_{\zeta},f_{\zeta^{-1}}\in \mathcal{H}(G,J,\lambda)$.  

Suppose that $w_x\not\in I_G(\lambda)$.   Then $f_{\zeta}\star f_{\zeta^{-1}}(1_G)=q^4.$
Furthermore, $ \supp(f_{\zeta}\star f_{\zeta^{-1}})=J$.  Hence $f_{\zeta}$ is an invertible element of $\mathcal{H}(G,\lambda)$.

Now, suppose that $w_x\in I_G(\lambda)$, then $w_y\in I_G(\lambda)$ by Lemma \ref{intchi}.   Hence $f_{w_x}, f_{w_y}\in\mathcal{H}(G,\lambda)$.  We have $f_{w_i}\star f_{w_i}(1_G)=[J:J\cap w_iJw_i]$ is a power of $q$, $i=1,2$.  By \cite[Lemma $5.12$]{St08}, $I_G(\eta)=J\GE J$, thus the support of $\mathcal{H}(G,\lambda)$ is contained in $J\GE J$.  Hence, $\supp(f_{w_i}\star f_{w_i})\subseteq (J\cup Jw_iJ)\cap J \GE  J=J((\P(\Lambda)\cup \P(\Lambda)w_i\P(\Lambda))\cap \GE )J$ by Stevens \cite[Lemma 2.6]{St08}.  Therefore $\supp(f_{w_i}\star f_{w_i})=J\cup Jw_i J$ and $f_{w_i}$, $i=1,2$, are invertible elements of $\mathcal{H}(G,\lambda)$. By Stevens \cite[Lemma $7.11$]{St08} we have $(J\cap N)^{w_x}\subseteq J\cap N$ and $(J\cap \overline{N})^{w_y}\subseteq J\cap \overline{N}$.  By the Iwahori decomposition of $J$,
\[Jw_xJw_yJ=J(w_x(J\cap\overline{N})w_x)w_xw_y(w_y(J\cap T)w_y)(w_y(J\cap N)w_y)J= Jw_xw_yJ.\]
Hence $f_{w_y}\star f_{w_x}$ is an invertible element of $\mathcal{H}(G,\lambda)$ supported on the single double coset $J\zeta J$.
\end{proof}

\subsection{Cuspidal representations}

The following theorem addresses the construction of all irreducible cuspidal $\ell$-modular and $\ell$-adic representations of $G$.
\begin{thm}\label{plcusps}
~
\begin{enumerate}
\item   Let $[\Lambda,n,0,\beta]$ be a skew semisimple stratum in $A$, $\theta\in\mathcal{C}_{-}(\beta,\Lambda)$, $\eta$ the unique Heisenberg representation containing $\theta$, $\kappa$ a $\beta$-extension of $\eta$ and $\sigma$ be an irreducible cuspidal representation of $\M(\Lambda_E)$.  Then $\I_{\kappa}(\sigma)$ is quasi-projective.  Furthermore, if $\P(\Lambda_E)$ is a maximal parahoric subgroup of $\GE$ then  $\I_{\kappa}(\sigma)$ is irreducible and cuspidal.
\item Let $\pi$ be an irreducible cuspidal representation of $G$.  Then there exist a skew semisimple stratum $[\Lambda,n,0,\beta]$ with $\P(\Lambda_E)$ a maximal parahoric subgroup of $\GE$, $\theta\in\mathcal{C}_{-}(\beta,\Lambda)$,  a $\beta$-extension $\kappa$ of the unique Heisenberg representation $\eta$ which contains $\theta$ and an irreducible cuspidal representation $\sigma$ of $\M(\Lambda_E)$ such that $\pi\simeq \I_{\kappa}(\sigma).$
\end{enumerate}
\end{thm}

\begin{proof}~
\begin{enumerate}
\item Quasi-projectivity follows \emph{mutatis mutandis} the proof given in Vign\'eras \cite[Proposition $6.1$]{vigneras}. So suppose $\P(\Lambda_E)$ is a maximal parahoric subgroup of $\GE$. By Theorem \ref{pl=lz} and Lemma \ref{lemma111} we have 
\begin{align*}\R_{\kappa}\circ \I_{\kappa}(\sigma)\simeq \R^E_{\Lambda}\circ\I^E_{\Lambda}(\sigma)\simeq \sigma.\end{align*}
The proof of irreducibility follows \emph{mutatis mutandis} the proof given in Vign\'eras \cite[Proposition $7.1$]{vigneras}. 
\item 
By Theorem \ref{plss}, $\pi$ contains a skew semisimple stratum $[\Lambda,n,0,\beta]$.    Suppose $\theta\in\mathcal{C}_{-}(\Lambda,\beta)$ is a skew semisimple character which $\pi$ contains. Let $\kappa$ be a $\beta$-extension of the unique Heisenberg representation $\eta$ which contains $\theta$.  Then $\pi$ contains $\kappa\otimes \sigma$ for some $\sigma\in\Irr_R(\M(\Lambda_E))$.  We show that we may assume that $\sigma$ is cuspidal.  If $\P(\Lambda_E)$ is not maximal then $\sigma$ is cuspidal, so we can suppose that $\P(\Lambda_E)$ is maximal.  Let $B(\Lambda_E)$ be the standard Borel subgroup of $\M(\Lambda_E)$ and $\P(\Lambda'_E)$ the preimage of $B(\Lambda_E)$ under the projection map.  Suppose that $r^{\M(\Lambda_E)}_{B(\Lambda_E)}(\sigma)\neq 0$.  Then, as $\pi$ contains $\sigma$, $r^{\M(\Lambda_E)}_{B(\Lambda_E)}(\R_{\kappa}(\pi))\neq 0$. We have
 \begin{align*}
 r^{\M(\Lambda_E)}_{B(\Lambda_E)}(\R_{\kappa}(\pi))&\simeq \Hom_{J^1}(\kappa,\pi)^{\P_1(\Lambda'_E)J^1/J^1}\\
 &\simeq \Hom_{\P_1(\Lambda'_E)J^1}(\kappa,\pi)
 \end{align*}
 which, by Lemma \ref{star}, implies that $\R_{\kappa',\Lambda'}(\pi)\neq 0$ where $\kappa'$ is the unique $\beta$-extension containing $\tau_{\Lambda,\Lambda',\beta}(\theta)$ compatible with $\kappa$.  Hence $\pi$ contains a skew semisimple stratum $[\Lambda',n,0,\beta]$ such that $\P(\Lambda'_E)$ is not maximal and thus contains $\kappa'\otimes\sigma'$ with $\sigma'$ a cuspidal representation of $\M(\Lambda'_E)$.  By Theorem \ref{jacquetcommutes} and Lemma \ref{plcovers} if  $\pi$ contains a skew semisimple stratum $[\Lambda,n,0,\beta]$ such that $\P(\Lambda_E)$ is not a maximal parahoric subgroup of $\GE$ then $\pi$ is not cuspidal.  Therefore $\P(\Lambda_E)$ is maximal and $\sigma$ is cuspidal.\end{enumerate}
\end{proof}

For level zero representations we can refine the exhaustive list of irreducible cuspidal representations given in Theorem \ref{plcusps} into a classification.

\begin{thm}
For $i=1,2$, let $\P(\Lambda_i)$ be maximal parahoric subgroups of $G$ and $\sigma_i$ be an irreducible cuspidal representation of $\M(\Lambda_i)$.  If $\Hom_G(\I_{\Lambda_1}(\sigma_1),\I_{\Lambda_2}(\sigma_2))\neq 0$ then $(\P(\Lambda_1),\sigma_1)$ and $(\P(\Lambda_2),\sigma_2)$ are conjugate.
\end{thm}

\begin{proof}
By reciprocity and Lemma \ref{lemma},
\[\Hom_G\left(\I_{\Lambda_1}(\sigma_1),\I_{\Lambda_2}(\sigma_2)\right)\simeq\bigoplus_{n\in D_{1,2}}\Hom_{\M(\Lambda_1)}\left(\sigma_1,{i}_{ P_{\Lambda_1,n\Lambda_2}}^{\M(\Lambda_1)} {\left(r^{\M(\Lambda_2)}_{P_{\Lambda_2,n^{-1}\Lambda_1}}(\sigma_2)\right)^n}\right).\]
Hence \[\Hom_G\left(\I_{\Lambda_1}(\sigma_1),\I_{\Lambda_2}(\sigma_2)\right)\neq 0\] if and only if there exists $n\in D_{1,2}$ such that 
\[\Hom_{\M(\Lambda_1)}\left(\sigma_1,{i}_{ P_{\Lambda_1,n\Lambda_2}}^{\M(\Lambda_1)} {\left(r^{\M(\Lambda_2)}_{P_{\Lambda_2,n^{-1}\Lambda_1}}(\sigma_2)\right)^n}\right)\neq 0.\]
Assume there exists such an element $n$. 
By cuspidality of $\sigma_2$, $P_{\Lambda_2,n^{-1}\Lambda_1}=\M(\Lambda_2)$; hence $\P_1(\Lambda_2)(\P(\Lambda_2)\cap \P(n^{-1}\Lambda_1))/\P_1(\Lambda_2)= M(\Lambda_2)$.  By cuspidality of $\sigma_1$, $P_{\Lambda_1,n\Lambda_2}=\M(\Lambda_1)$; hence $\P_1(\Lambda_1)(\P(\Lambda_1)\cap \P(n\Lambda_2))/\P_1(\Lambda_1)= \M(\Lambda_1)$.  If $\P(\Lambda_1)$ and $\P(\Lambda_2)$ are not conjugate then for all $g\in G$, in particular $n\in D_{1,2}$, the group $\P(\Lambda_1)\cap \P(g\Lambda_2)$ must stabilise an edge in the building and hence is not maximal.  Thus it cannot surject onto either $\M(\Lambda_1)$ or $\M(\Lambda_2)$.  Hence there exists $n\in D_{1,2}$ such that $\P(\Lambda_1)=\P(n\Lambda_2)$ and 
\[\Hom_{\M(\Lambda_1)}\left(\sigma_1,\sigma^n_{2}\right)\neq \{0\},\]
i.e. $(\P(\Lambda_1),\sigma_1)$ and $(\P(\Lambda_2),\sigma_2)$ are conjugate.
\end{proof}

\begin{remark}\label{notlift}
Let $\ell\mid (q^2-q+1)$.  The irreducible cuspidal $\ell$-modular representations $\I_{\Lambda_x}(\overline{\tau}^+(\overline{\chi}))$ do not lift.   A lift must necessarily be cuspidal as the Jacquet functor commutes with reduction modulo-$\ell$.  However, by Theorem \ref{plcusps}, all $\ell$-adic level zero irreducible cuspidal representations are of the form $\I_{\Lambda_x}(\sigma_x)$ or $\I_{\Lambda_y}(\sigma_y)$ with $\sigma_x$ (resp. $\sigma_y$) an irreducible cuspidal $\ell$-adic representation of $\M(\Lambda_x)$ (resp. $\M(\Lambda_y)$).  Furthermore, $r_{\ell}(\I_{\Lambda_w}(\sigma_w))=\I_{\Lambda_w}(r_{\ell}(\sigma_w))$  as compact induction commutes with reduction modulo-$\ell$, for $w\in\{x,y\}$.  Hence, by Section \ref{u3finitecusplmod}, $\I_{\Lambda_x}(\overline{\tau}^+(\overline{\chi}))$ does not lift, but does appear in the reduction modulo-$\ell$ of $\I_{\Lambda_x}(\sigma(\psi))$ where 
$r_{\ell}(\I_{\Lambda_x}(\tau(\psi))=\I_{\Lambda_x}(\overline{\nu}(\overline{\chi}))\oplus\I_{\Lambda_x}(\overline{\tau}^+(\overline{\chi}))$.
\end{remark}

\section{Parabolically induced representations}
Let $\omega_{F/F_0}$ be the unique character of $F_0^\times$ associated to $F/F_0$ by local class field theory.  That is, $\omega_{F/F_0}$ is defined by $\omega_{F/F_0}\mid_{\mathfrak{o}_{F_0}^\times}=1$ and $\omega_{F/F_0}(\varpi_F)=-1$.  All extensions of $\omega_{F/F_0}$ to $F^\times$ take values in $\Zl^\times$, hence are integral.  Let $\chi_1$ be a character of $F^\times$ and $\chi_2$ be a character of $F^1$.   Let $\delta_B$ be the character of $T$ given by $\delta_B(\diag(x,y,\overline{x}^{-1}))=\left|x\right|_{F}^{4}$, i.e. the character $\chi_1(x)=|x|^4_F$, which we also think of as a character of $F^\times$.  Because the image of $\delta_B$ is contained in $\Zl^\times$, $\delta_B$ is integral.  If $q^2\equiv 1\bmod{\ell}$ then $\delta_B$ is trivial.  

Let $\chi$ be the character of $T$ defined by
\[\chi\left(\diag(x,y,\overline{x}^{-1})\right)=\chi_1(x)\chi_2(x\overline{x}^{-1}y)\]
which is well defined because $x\mapsto x\overline{x}^{-1}$ is a surjective map $F^\times\rightarrow F^1$.  Every character of $T$ appears in this way; we can recover $\chi_1$ and $\chi_2$ from $\chi$
\[\chi_1(x)=\chi(\diag(x,\overline{x}/x,\overline{x}^{-1})),\quad \chi_2(y)=\chi(\diag(1,y,1)).\]
The character $\chi_2$ factors through the determinant and
\[i_B^G(\chi)\simeq i_B^G (\chi_1)(\chi_2\circ\det)\]
where $\chi_1$ is considered the character $\chi_1(\diag(x,y,\overline{x}^{-1}))=\chi_1(x)$ of $T$.  Hence the reducibility of $i_B^G(\chi)$ is completely determined by that of $i_B^G (\chi_1)$.  The character $\chi$ is not regular if $\chi_1(x)=\chi_1(\overline{x})^{-1}$ which occurs if and only if $\chi_1$ is an extension of $1$ or $\omega_{F/F_0}$ to $F^\times$.   An irreducible character $\chi$ has level zero if and only if both $\chi_1$ and $\chi_2$ have level zero.

\subsection{Hecke Algebras}

To find the characters $\chi$ such that the induced representation $i_B^G(\chi)$ is reducible we study the algebras $\mathcal{H}(G,\lambda)$.
\begin{thm}\label{Heckealgdesc}
Suppose $\lambda_T$ is a character of $T^0$.    Let $(J,\lambda)$ be a $G$-cover of $(T^0,\lambda_T)$ as constructed in Lemma \ref{plcovers}.
\begin{enumerate}
\item If $\lambda_T$ is regular then $\mathcal{H}(G,\lambda)\simeq R[X^{\pm 1}]$.
\item If $\lambda_T$ is not regular then $\mathcal{H}(G,\lambda)$ is a two-dimensional algebra generated as an $R$-algebra by $f_{w_x}$ and $f_{w_y}$ and the relations
\begin{align*}
f_{w_x}\star f_{w_x}=(q^a-1)f_{w_x}+q^a; \qquad f_{w_y}\star f_{w_y}=(q-1)f_{w_y}+q;
\end{align*}
where $a=3$ and $f_{w_x}(1)=f_{w_y}(1)=1$ if $\lambda_T$ is trivial on $T^1$ and factors through the determinant, and $a=1$, $f_{w_x}(1)=\frac{1}{q}$ and $f_{w_y}(1)=1$ if not.
\end{enumerate}
\end{thm}

\begin{proof}
If $g\in I_G(\lambda)$ then $\I_g(\lambda)\simeq R$ because $\chi$ is a character.  For $g\in I_G(\lambda)$, $r\in R$ we let $f_{g,r}$ denote the unique function supported on $JgJ$ with $f_{g,r}(1_G)=r$.  If $\lambda_T$ is regular then the support of $\mathcal{H}(G,\lambda)$ is $JTJ=\bigcup_{n\in\mathbb{Z}}J\zeta^n J$ and, since each intertwining space is one-dimensional and $f_{\zeta^n,1}$ has support $J\zeta^n J$, we have an isomorphism $\mathcal{H}(G,\lambda)\simeq R[X^{\pm 1}]$ defined by $f_{\zeta,1}\mapsto X$.

Suppose $w_x\in I_G(\lambda)$. By Lemma \ref{intchi}, $w_x$ intertwines $\lambda$ if and only if $w_y$ intertwines $\lambda$.  The support of the Hecke algebra is contained in the intertwining of $\eta=\Res^J_{J^1}(\kappa)$ which is $J\GE J$.  By the semisimple intersection property, Stevens \cite[Lemma $2.6$]{St08}, and the Bruhat decomposition we have $J\GE J=\bigcup_{w\in \widetilde{W}}JwJ$.  As in the proof of lemma \ref{plcovers} we have $Jw_xJw_yJ=Jw_xw_yJ$ and, similarly, $Jw_yJw_xJ=Jw_yw_xJ$.  Hence, as the intertwining spaces are one-dimensional, the support of $f_{w_x}\star f_{w_y}\star f_{w_x}\star\cdots\star f_{w_i}$ is $Jw_xw_yw_x\cdots w_iJ$. Thus, as $\widetilde{W}$ is an infinite dihedral group generated by $w_x$ and $w_y$, $\mathcal{H}(G,\lambda)$ is generated by $f_{w_x,1}$ and $f_{w_y,1}$ and the quadratic relations $f_{w_x,1}\star f_{w_x,1}$ and $f_{w_y,1}\star f_{w_y,1}$.  Let $\Lambda^x$ and $\Lambda^y$ be $\mathfrak{o}_E$-lattice sequences such that the parahoric subgroups $\P(\Lambda^x_E)=\P(\Lambda_E)\cup\P(\Lambda_E)w_x\P(\Lambda_E)$ and $\P(\Lambda^y_E)=\P(\Lambda_E)\cup\P(\Lambda_E)w_y\P(\Lambda_E)$. The parahoric subgroups $\P(\Lambda^x_E)$ and $\P(\Lambda_E^y)$ are non-conjugate, maximal and contain $\P(\Lambda_E)$. Let $\kappa_x$ and $\kappa_y$ be the $\beta$-extensions, compatible with $\kappa$, defined by Lemma \ref{compatbeta} related to the skew semisimple strata $[\Lambda^x,n,0,\beta]$ and $[\Lambda^y,n,0,\beta]$.  

For $z\in\{x,y\}$, let $\widehat{\kappa}_z=\Res^J_{J^1(\beta,\Lambda^i)\P(\Lambda_E)}(\kappa_z)$.  We have a support preserving isomorphism
\[\mathcal{H}(G,\kappa\otimes\sigma)\simeq \mathcal{H}(G,\widehat{\kappa}_z\otimes\sigma)\]
by Lemma \ref{compatbeta} and transitivity of compact induction. We have a support preserving injection of algebras
\[\mathcal{H}(\P(\Lambda_E^z),\sigma)\rightarrow\mathcal{H}(\P(\Lambda^z),\widehat{\kappa}_z\otimes \sigma)\]
defined by $\Phi\mapsto \widehat{\kappa}_z\otimes\Phi$ where $\sigma$ is considered as a character of $\P(\Lambda_E)$ trivial on $\P_1(\Lambda_E)$. 

Let $B_z$ be the standard Borel subgroup of $\M(\Lambda_E^z)$.  In the $\ell$-adic case, by \cite[Theorem $4.14$]{HL80}, if $i^{\M(\Lambda^z_E)}_{B_z}(\overline{\sigma})=\rho^z_1\oplus\rho^z_2$ with $\dim(\rho^z_1)\geqslant \dim(\rho^z_2)$ then $\mathcal{H}(\M(\Lambda^z_E),\overline{\sigma})$ is generated by $T^z_w$ which is supported on the double coset $B_zw_xB_z$ and satisfies the quadratic relation
\begin{align*}T^z_w\star T^z_w= (d_z-1)T^z_w+d_z T^z_1\end{align*}
where $d_z=\dim(\rho^z_1)/\dim(\rho^z_2)$ and $T^z_1$ is the identity of $\mathcal{H}(\M(\Lambda^z_E),\overline{\sigma})$.  By Section \ref{cuspU11U21}, $d_y=q$ and
\begin{align*}
d_x=\begin{cases}
q^3&\text{if }\lambda_T \text{ is trivial on }T^1\text{ and factors through the determinant};\\
q&\text{otherwise.}
\end{cases}
\end{align*}

In the $\ell$-modular case, we choose a lift $\widehat{\overline{\sigma}}$ of $\overline{\sigma}$ such that $\widehat{\overline{\sigma}}^{w_x}=\widehat{\overline{\sigma}}$.  Let $L$ be a lattice in $\widehat{\overline{\sigma}}$.  Recall that $\widehat{\overline{\sigma}}$ is called a reduction stable of $\overline{\sigma}$ if $\mathcal{H}(\M(\Lambda_E^z),\overline{\sigma})=\Zl\otimes_{\Fl}\mathcal{H}(\M(\Lambda_E^z),L)$ and $\mathcal{H}(\M(\Lambda_E^z),\widehat{\overline{\sigma}})=\Ql\otimes_{\Fl}\mathcal{H}(\M(\Lambda_E^z),L)$.  A basis of $\mathcal{H}(\M(\Lambda_E^z),\widehat{\overline{\sigma}})$ is called reduction stable if it is a basis of $\mathcal{H}(\M(\Lambda_E^z),L)$ and $\widehat{\overline{\sigma}}$ is reduction stable. By \cite[Section 3.1]{GHM}, $\widehat{\overline{\sigma}}$ is reduction stable and a basis of $\mathcal{H}(\M(\Lambda_E^z),\widehat{\overline{\sigma}})$ is reduction stable. Hence we obtain a basis of $\mathcal{H}(\M(\Lambda_E^z),\overline{\sigma})$ satisfying the quadratic relations required by reduction modulo-$\ell$. 

By inflation $T^z_w$ determines an element $f_{w_z,r_z}\in \mathcal{H}(\P(\Lambda_E^z),\sigma)$ supported on $Jw_zJ$.  Furthermore, $f_{w_x,1}\star f_{w_x,1}(1_G)=[J:J\cap w_xJw_x]=q^3$ and $f_{w_y,1}\star f_{w_y,1}(1_G)=[J:J\cap w_yJw_y]=q$ in all cases, hence $r_x=r_y=1$ if $\lambda_T$ is trivial on $T^1$ and factors through the determinant, and $r_x=\frac{1}{q}$ and $r_y=1$ otherwise.
\end{proof}

\subsection{Reducibility points}\label{redpoints}
Suppose $i^G_B(\chi)$ is reducible and let $\lambda_T=\Res^T_{T^0}(\chi)$.  By Theorem \ref{Heckealgdesc}, $\lambda_T$ is not regular.  Let $(J,\lambda)$ be a $G$-cover of $(T^0,\lambda_T)$ as constructed in Lemma \ref{plcovers} with $\lambda=\kappa\otimes\sigma$.  If $\pi$ is an irreducible quotient of $\I_{\kappa}(\sigma)$ and an irreducible quotient of $i^G_B(\chi)$ then, by exactness of the Jacquet functor, $r^G_B(\pi)$ is one-dimensional.  Hence, as $(j_P)^*(M_\lambda(\pi))\simeq M_{\lambda_T}(r^G_P(\pi))$ by Theorem \ref{jacquetcommutes}, $\pi$ must correspond to a character of $\mathcal{H}(G,\lambda)$ under the bijection of Theorem \ref{qplemma}.  The characters of $\mathcal{H}(G,\lambda)$ are determined by their values on the generators $f_{w_x}$ and $f_{w_y}$.  Let $a$ be given by Theorem \ref{Heckealgdesc}.  The characters of $\mathcal{H}(G,\lambda)$ are summarised in the following table.

\begin{center}
\begin{tabular}{QQQ}
\toprule
\text{Character}\text{ of }\mathcal{H}_R(G,\lambda) & \text{Value on }f_{w_x} & \text{Value on }f_{w_y}\\ \midrule
 \Xi_{\text{sgn}}&-1&-1\\ \midrule
 \Xi_{\text{ind}} & q^a & q\\ \midrule
 \Xi_1           &  q^a & -1 \\ \midrule
 \Xi_2            &  -1  & q\\\bottomrule
\end{tabular}
\end{center}

 If $q^a\neq-1\bmod{\ell}$, then these characters are distinct; if $q^a=-1\bmod{\ell}$, but $q\neq -1\bmod{\ell}$, then there are two characters $ \Xi_{\text{sgn}}=\Xi_1$ and $\Xi_{\text{ind}} =\Xi_2$; if $q=-1\mod{\ell}$ then there is a unique character $ \Xi_{\text{sgn}}=\Xi_1=\Xi_{\text{ind}} =\Xi_2$.  To calculate the values of $\chi$ where this reducibility occurs we make the injection $j_B:\mathcal{H}(T,\lambda_T)\rightarrow\mathcal{H}(G,\lambda)$ explicit.  We have $\mathcal{H}(T,\lambda_T)\simeq R[X^{\pm 1}]$ and the injection is induced by mapping $X$ to $\varepsilon f_{w_x}f_{w_y}$ for some scalar $\varepsilon\in R$.  It is determining the sign of this scalar which requires work. The normalised restriction map $(j_B)^*$ is then induced by this injection and twisting by $\delta_B^{-\frac{1}{2}}$. As the space of constant functions forms an irreducible subrepresentation of $i^G_B(\delta_B^{-\frac{1}{2}})$; we find that $\varepsilon=1$ if $\lambda_T$ is trivial on $T^1$ and factors through the determinant.  Using an alternative method, Keys \cite{keys} computed the $\ell$-adic reducibility points.  By reduction modulo-$\ell$ from the results of Keys we find $\varepsilon=-q$ in all other cases.

\begin{thm}\label{redpoints}
Let $\chi$ be an irreducible $\ell$-modular character of $T$.   Then $i^G_B(\chi)$ is reducible exactly in the following cases:
\begin{enumerate}
\item $\chi=\delta_B^{\pm \frac{1}{2}}$;
\item $\chi= \eta\delta_B^{\pm \frac{1}{4}}$ where $\eta$ is any extension of $\omega_{F/F_0}$ to $F^\times$;
\item $\chi_1$ is nontrivial, but $\chi_1\mid_{F_0^\times}$ is trivial.
\end{enumerate}
\end{thm}

\subsection{Parahoric restriction and parabolic induction}
As the parabolic functors respect the decomposition of $\mathfrak{R}_R(G)$ by level, by Vign\'eras \cite[II 5.12]{Vig96}, if $\chi$ is a level zero character of $T$ (i.e. a character of $T$ trivial on $T^1$) then all irreducible subquotients of $i^G_B(\chi)$ have level zero.  

\begin{lemma}\label{theorem2}
Let $w\in\{x,y\}$ and $\chi$ be a level zero character of $T$.  Then
$\R_{\Lambda_w}(i_B^G(\chi))\simeq {i}_{B_w}^{\M(\Lambda_w)}(\chi)$.\end{lemma}

\begin{proof}
The proof follows by Mackey theory as the maximal parahoric subgroups of $G$ satisfy the Iwasawa decomposition.  
\end{proof}

Let $[\Lambda,n,0,\beta]$ be a skew semisimple stratum in $A$.  Let $\theta\in\mathcal{C}_{-}(\Lambda,\beta)$, and $\kappa$ be a $\beta$-extension of the unique Heisenberg representation containing $\theta$.   Let $\chi$ be an irreducible $\ell$-modular character of $T$ which contains the $R$-type $(J_T,\kappa_T\otimes\sigma)$.  Futhermore, suppose that $(J,\kappa\otimes\sigma)$ is a $G$-cover of $(J_T, \kappa_T\otimes\sigma)$ relative to $B$, as in Lemma \ref{plcovers}.  Let $\Lambda^{m}$ be an $\mathfrak{o}_E$-lattice sequence in $V$ such that $\P(\Lambda^{m}_E)$ is maximal and $\P(\Lambda_E)\subset \P(\Lambda^{m}_E)$.  Let $\theta_{m}=\tau_{\Lambda,\Lambda^{m},\beta}(\theta)$ and $\kappa_{m}$ be the unique $\beta$-extension of the unique Heisenberg representation containing $\theta_{m}$ which is compatible with $\kappa$, as in Lemma \ref{compatbeta}.  Let $B(\Lambda_E^m)$ be the Borel subgroup of $\M(\Lambda_E^m)$ whose preimage under the projection map $\P(\Lambda_E^m)\rightarrow \M(\Lambda_E^m)$ is equal to $J$.  Suppose $B(\Lambda_E^m)$ has Levi decomposition $B(\Lambda_E^m)=T(\Lambda_E^m)\ltimes N(\Lambda_E^m)$.

The next theorem is a generalisation of a weakening of Lemma \ref{theorem2}, precisely it generalises the isomorphism Theorem \ref{theorem2} induces in the Grothendieck group $\mathfrak{Gr}_R(\M(\Lambda))$. %

\begin{thm}\label{SZadapt}We have
\begin{align*}
\left[\R_{\kappa_{m}}( i^G_B(\chi))\right]\simeq\left[ i^{M(\Lambda_E^{m})}_{B(\Lambda_E^m)}(\R_{\kappa_T}(\chi))\right].
\end{align*}
\end{thm}

\begin{proof}
We prove the corresponding result in the $\ell$-adic case first and deduce the $\ell$-modular result by reduction modulo-$\ell$.  The proof in the $\ell$-adic case follows a similar argument made for $\GL_n(F)$ in Schneider--Zink \cite{SZ99}. Let $\Omega_T=[T,\rho]_T$ and $\Omega=[T,\rho]_G$ be inertial equivalence classes. Let $\mathfrak{R}_{\Ql}(\Omega)$ denote the full subcategory of $\R_{\Ql}(G)$ of representations all of whose irreducible subquotients have inertial support in $\Omega$, and $\mathfrak{R}_{\Ql}(\Omega_T)$ denote the full subcategory of $\R_{\Ql}(T)$ of representations all of whose irreducible subquotients have inertial support in $\Omega_T$.   Let $\omega$ denote the $\M(\Lambda^{m}_E)$-conjugacy class of $\sigma$ and $\omega_T$ the $T(\Lambda_E^m)$-conjugacy class of $\sigma$.  Let $\mathfrak{R}_{\Ql}(\omega)$ be the full subcategory of $\mathfrak{R}_{\Ql}(\M(\Lambda^{m}_E))$ of representations all of whose irreducible subquotients have supercuspidal support in $\omega$ and  $\mathfrak{R}_{\Ql}(\omega_T)$ be the full subcategory of $\mathfrak{R}_{\Ql}(T(\Lambda_E^m))$ of representations all of whose irreducible subquotients have lie in $\omega_T$.  Let $M_{\omega}:\mathfrak{R}_{\Ql}(\omega)\rightarrow \mathcal{M}(\M(\Lambda_{E}^{m}),\sigma)$ be defined by $\rho\mapsto\Hom_{B(\Lambda_E^m)}(\sigma,\rho)$, for $\rho\in\mathfrak{R}_{\Ql}(\omega)$.  Similarly, let $M_{\omega_T}:\mathfrak{R}_{\Ql}(\omega_T)\rightarrow \mathcal{M}(T(\Lambda_E^m),\sigma)$ be defined by $\rho\mapsto\Hom_{T(\Lambda_E^m)}(\sigma,\rho)$, for $\rho\in\mathfrak{R}_{\Ql}(\omega_T)$.  We prove that the following diagram commutes.
 \begin{center}
\begin{tikzpicture}
\matrix [matrix of math nodes, column sep=1cm, row sep=2.2cm,text height=1.5ex, text depth=0.25ex]
{
&\node(l1) {\mathfrak{R}_{\Ql}(\omega)}; &\node(l2) {\mathcal{M}(\M(\Lambda_E^{m}),\sigma)};&\node(l4) {\mathfrak{R}_{\Ql}(\omega)}; & \\
\node(m1) {\mathfrak{R}_{\Ql}(\Omega)}; &\node(m2) {\mathcal{M}(G,\kappa\otimes\sigma)};& &\node(m3) {\mathcal{M}(T(\Lambda_E^m),\sigma)}; &\node(m4) {\mathfrak{R}_{\Ql}(\omega_T)}; \\
&\node(t1) {\mathfrak{R}_{\Ql}(\Omega_T)}; &\node(t2) {\mathcal{M}(T,\kappa_T\otimes\sigma)};&\node(t4) {\mathfrak{R}_{\Ql}(\Omega_T)};&\node(t3) {};  \\
};
\draw [->] (t1) -- (t2)
node [above,midway] {$\simeq$}
node [below,midway] {$M_{\kappa_T\otimes \sigma}$};
\draw [->] (t4) -- (t2)
node [above,midway] {$\simeq$}
node [below,midway] {$M_{\kappa_T\otimes \sigma}$};
\draw [->] (m1) -- (m2)
node [above,midway] {$\simeq$}
node [below,midway] {$M_{\kappa\otimes \sigma}$};
\draw [->] (m4) -- (m3)
node [above,midway] {$\simeq$}
node [below,midway] {$M_{\omega_T}$};
\draw [->] (t1) -- (m1)
node [left,midway] {$i^G_B$\phantom{y}};
\draw [->] (t2) -- (m2)
node [right,midway] {\phantom{y}$(j_B)_*$};
\draw [->] (t2) -- (m3)
node [left,midway] {$\Res$\phantom{y}};
\draw [->] (t4) -- (m4)
node [right,midway] {\phantom{y}$\R_{\kappa_T}$};
\draw [->] (l1) -- (l2)
node [above,midway] {$\simeq$}
node [below,midway] {$M_{\omega}$};
\draw [->] (l4) -- (l2)
node [above,midway] {$\simeq$}
node [below,midway] {$M_{\omega}$};
\draw [->] (m1) -- (l1)
node [left,midway] {$\R_{\kappa_{m}}$\phantom{i}};
\draw [->] (m2) -- (l2)
node [right,midway] {\phantom{y}$\Res$};
\draw [->] (m3) -- (l2)
node [left,midway] {$(j_{B(\Lambda_E^m)})_*$\phantom{i}};
\draw [->] (m4) -- (l4)
node [right,midway] {\phantom{y}$i^{\M(\Lambda^{m}_E)}_{B(\Lambda_E^m)}$};
\end{tikzpicture}
\end{center}
We have $M_{\omega}\circ i^{\M(\Lambda_E^{m})}_{B(\Lambda_E^m)}\simeq (j_{B(\Lambda_E^m)})_*\circ M_{\omega_T}$ and $M_{\kappa\otimes\sigma}\circ i^G_B\simeq (j_B)_*\circ M_{\kappa_T\otimes\sigma}$ by Bushnell--Kutzko \cite[Corollary 8.4]{BK98}.  

We have support preserving injections $\alpha_1:\mathcal{H}(\M(\Lambda^{m}_E),\sigma)\rightarrow \mathcal{H}(G,\kappa\otimes \sigma)$ and $\alpha_2:\mathcal{H}(T(\Lambda_E^m),\sigma)\rightarrow \mathcal{H}(T,\kappa_T\otimes \sigma)$, as in the proof of Theorem \ref{Heckealgdesc}, hence restriction functors $\mathcal{M}(G,\kappa\otimes \sigma)\rightarrow \mathcal{M}(\M(\Lambda^{m}_E),\sigma)$ and $\mathcal{M}(T,\kappa_T\otimes \sigma)\rightarrow \mathcal{M}(T(\Lambda_E^m),\sigma)$, denoted in the diagram by $\Res$.  Because $\mathcal{H}(T(\Lambda_E^m),\sigma)$ is one-dimensional and the injections defined are homomorphisms of algebras we must have $j_B\circ\alpha_1\simeq j_{B(\Lambda_E^m)}\circ\alpha_2$, hence also $\Res\circ(j_B)_*\simeq (j_{B(\Lambda_E^m)})_*\circ \Res$. 

We show that $M_{\omega}\circ\R_{\kappa_{m}}\simeq \Res\circ M_{\kappa\otimes\sigma}$, a similar argument shows that $M_{\omega_T}\circ\R_{\kappa_T}\simeq \Res\circ M_{\kappa_T\otimes \sigma}$.  Let $\pi\in\mathfrak{R}_{\Ql}(\Omega)$.  By Lemma \ref{star} and adjointness, we have
\begin{align*}
M_{\omega}(\R_{\kappa_{m}}(\pi))&=\Hom_{B(\Lambda_E^m)}(\sigma,\R_{\kappa_{m}}(\pi))\\
&=\Hom_{B(\Lambda_E^m)}(\sigma,\R_{\kappa_{m}}(\pi))\\
&\simeq\Hom_{J}(\sigma,(\R_{\kappa_{m}}(\pi))^{J^1_{m}\P_1(\Lambda_E)/J^1_{m}})\\
&\simeq\Hom_{J}(\sigma,\R_{\kappa}(\pi))\\
&\simeq\Hom_{J^1}(\kappa\otimes\sigma,\pi)=M_{\kappa\otimes\sigma}(\pi).
\end{align*}

In the $\ell$-modular case, we choose lifts of $\kappa$ and $\chi$ and then by the $\ell$-adic isomorphism and reduction modulo-$\ell$ we have $\left[\R_{\kappa_{m}}( i^G_B(\chi))\right]\simeq\left[ i^{M(\Lambda_E^{m})}_{B(\Lambda_E^m)}(\R_{\kappa_T}(\chi))\right]$.  
\end{proof} 

\subsection{Parabolic induction, $\kappa$-restriction, and covers}\label{sect64}

Let $\chi$ be an irreducible character of $T$.  Let $(T^0,\lambda_T)$ be an $R$-type contained in $\chi$ such that $(J,\lambda)$ is a $G$-cover of $(T^0,\lambda_T)$ relative to $B$ as constructed in Lemma \ref{plcovers} with $\lambda=\kappa\otimes \sigma$ and $\lambda_T=\kappa_T\otimes\sigma$ where $\kappa_T=\Res^J_{T^0}(\kappa)$.  Hence $J=\P(\Lambda_E)J^1$ with $\P(\Lambda_E)$ a non-maximal parahoric subgroup of $\GE$ corresponding to the $\mathfrak{o}_E$-lattice sequence $\Lambda_E$.  In all cases, there are two non-conjugate maximal parahoric which contain $\P(\Lambda_E)$; we denote the $\mathfrak{o}_E$-lattice sequences that correspond to these by $\Lambda^x_E$ and $\Lambda_E^y$.  Let $m\in\{x,y\}$ and let $(\kappa_m,\Lambda^m_E)$ be the unique pair compatible with $(\kappa,\Lambda_E)$ as in Lemma \ref{compatbeta}. 

\begin{lemma}\label{lemma64}
Let $\pi$ be an irreducible subrepresentation or quotient of $i^G_B(\chi)$ and $m\in\{x,y\}$.  Then $\R_{\kappa_{m}}(\pi)\neq 0$.
\end{lemma}

\begin{proof}
By the geometric lemma, $r^G_B(i^G_B(\chi))$ is filtered by $\chi$ and $\chi^{w_x}=\psi\chi$ for some unramified character $\psi$.  Hence, by exactness of the Jacquet functor, $r^G_B(\pi)=\psi\chi$. By Theorem \ref{jacquetcommutes}, $\pi$ contains $(J,\lambda)$ if and only if $r^G_B(\pi)$ contains $(T^0,\lambda_T)$.     Thus $\pi$ contains $(J,\lambda)$ hence $\R_{\kappa}(\pi)\neq 0$. Therefore $\R_{\kappa_{m}}(\pi)\neq 0$.  \end{proof}

The next lemma is crucial in our proof of unicity of supercuspidal support.  It shows that parabolic induction  preserves the semisimple character up to transfer.

\begin{lemma}\label{lemmacus}
Suppose that $i^G_B(\chi)$ has an irreducible cuspidal subquotient $\pi$.  Then there exists $m\in\{x,y\}$ such that $\R_{\kappa_m}(\pi)\neq 0$.
\end{lemma}

\begin{proof}
By Theorem \ref{plcusps} there exist a skew semisimple stratum $[\Lambda',n',0,\beta']$ such that $\P(\Lambda'_{E'})$ is a maximal parahoric subgroup of $G_{E'}$ where $G_{E'}$ denotes the $G$-centraliser of $\beta'$, a semisimple character $\theta'\in\mathcal{C}_{-}(\Lambda',\beta')$, a $\beta'$-extension $\kappa'$ to $J'=J(\Lambda',\beta')$ of the unique Heisenberg representation $\eta'$ containing $\theta'$ and a cuspidal representation $\sigma'\in\Irr(J'/(J')^1)$ such that $\pi\simeq \I_{\kappa'}(\sigma')$.

In Lemma \ref{plcovers}, the covers we constructed fall into two classes according to whether
\begin{enumerate}
\item $[\Lambda,n,0,\beta]$ is a scalar skew simple stratum in $A$, or
\item $[\Lambda,n,0,\beta]$ is a skew semisimple stratum in $A$ with splitting $V=V_1\oplus V_2$ with $V_1$ one-dimensional and $V_2$ two dimensional hyperbolic, $\beta=\beta_1\oplus\beta_2$ with $\beta_1$ and $\beta_2$ skew scalars, $G_E\simeq \U(1,1)(F/F_0)\times\U(1)(F/F_0)$, and $\P(\Lambda_{2,E})$ an Iwahori subgroup of $\U(1,1)(F/F_0)$.
\end{enumerate}

As $\pi$ contains $\kappa'\otimes\sigma'$, the restriction of $i^G_B(\chi)$ to $J'$ has $\kappa'\otimes\sigma'$ as a subquotient.  We choose $\widehat{\chi}$ an $\ell$-adic character lifting $\chi$ such that $i^G_B(\widehat{\chi})$ is reducible.  Then, because restriction and parabolic induction commute with reduction modulo-$\ell$,  the restriction of $i^G_B(\widehat{\chi})$ to $J'$ has an irreducible subquotient $\delta$ such that $r_{\ell}(\delta)$ contains $\kappa'\otimes \sigma'$. On restricting to $(J')^1$ we see that $\delta$ contains the unique lift $\widehat{\eta'}$ of $\eta'$ and, since $\delta$ is irreducible and $J'$ normalises $\widehat{\eta'}$, $\Res^{J'}_{(J')^1}(\delta)$ is a multiple of $\eta'$.  Thus $\delta=\widehat{\kappa'}\otimes \xi$ with $\widehat{\kappa'}$ a lift of $\kappa'$ and $\xi$ an irreducible representation of $J'/(J')^1$ whose reduction modulo-$\ell$ contains $\sigma$.  However, $\xi$ cannot be cuspidal; otherwise $i^G_B(\widehat{\chi})$ would have a cuspidal subquotient $\I_{\widehat{\kappa'}}(\xi)$.  Hence $G_{E'}$ is not compact. Therefore $[\Lambda',n',0,\beta']$ is also either a scalar skew simple stratum or a skew semisimple stratum with splitting $V=V'_1\oplus V'_2$ with $V'_1$ one-dimensional and $V'_2$ two-dimensional hyperbolic.  (Note that, as $\sigma$ is cuspidal non-supercuspidal, we must have $\ell\mid q+1$ or $\ell\mid q^2-q+1$ by Section \ref{cuspU11U21}.)

We continue by induction on the level $l(\pi)$ of $\pi$.  

The base step is when $\pi$ has level zero.  If $\pi$ has level zero then, as all subquotients of $i^G_B(\chi)$ have the same level as $\chi$ by Vign\'eras \cite[5.12]{Vig96}, $\chi$ and $i^G_B(\chi)$ have level zero.  Thus we can choose, and assume that we have chosen, $\kappa'$, $\kappa$ and $\kappa_T$ to be trivial. By conjugating, we may assume $\Lambda'=\Lambda_m$ for some $m\in\{x,y\}$ and then $\kappa_m=\kappa'$ is trivial and $\R_{\kappa_m}(\pi)=\R_{\Lambda_m}(\pi)\neq 0$. 

Suppose first that $[\Lambda,n,n-1,\beta]$ is equivalent to a scalar stratum $[\Lambda,n,n-1,\gamma]$.  The stratum $[\Lambda,n,n-1,\gamma]$ corresponds to a character $\psi_{\gamma}$ of $U_n(\Lambda)\cap G$ which extends to a character $\phi\circ\det$ of $G$.  Twisting by $\phi^{-1}\circ \det$ we reduce the level of $\pi$ and the level of $i^G_B(\chi)$.   The stratum $[\Lambda,n,n-1,\beta-\gamma]$ is equivalent to a semisimple stratum $[\Lambda,n,n-1,\alpha]$ and the representations  $\kappa(\phi^{-1}\circ\det)$, $\kappa_T(\phi^{-1}\circ\det)$ and $\kappa_m(\phi^{-1}\circ\det)$ for $m\in\{x,y\}$ are $\alpha$-extensions defined on the relevant groups.  Similarly, the stratum $[\Lambda',n',n'-1,\beta'-\gamma]$ is equivalent to a semisimple stratum $[\Lambda,n',n'-1,\alpha']$ and $\kappa'(\phi^{-1}\circ\det)$ is a $\alpha'$-extension.  Moreover, $\kappa_m(\phi^{-1}\circ\det)$ is compatible with $\kappa(\phi^{-1}\circ\det)$ for $m\in\{x,y\}$, $(\kappa\otimes\sigma)(\phi^{-1}\circ\det)$ is a $G$-cover of  $(\kappa_T\otimes\sigma)(\phi^{-1}\circ\det)$ relative to $B$,  $(\kappa_T\otimes\sigma)(\phi^{-1}\circ\det)$ is contained in $\chi(\phi^{-1}\circ\det)$, and $(\kappa'\otimes\sigma')(\phi^{-1}\circ\det)$ is contained in $\pi(\phi^{-1}\circ\det)$.  Thus, by induction, we have
\begin{align*}
 \R_{\kappa_m}(\pi)\simeq  \R_{\kappa_m(\phi^{-1}\circ\det)}(\pi(\phi^{-1}\circ\det))
\end{align*}
is non-zero for some $m\in\{x,y\}$.

Suppose now that $[\Lambda',n',n'-1,\beta']$ is equivalent to a scalar stratum $[\Lambda',n',n'-1,\gamma']$.  As in the last case, we can twist by a character to reduce the level.

Hence we may assume that both $[\Lambda,n,n-1,\beta]$ and $[\Lambda',n',n'-1,\beta']$ are not equivalent to scalar simple strata.  This forces $[\Lambda,n,0,\beta]$ (resp. $[\Lambda',n',0,\beta']$) to be semisimple (non-simple) with splitting $V=V_1\oplus V_2$ (resp. $V=V'_1\oplus V'_2$) with $V_1$ (resp. $V'_1$) one-dimensional and $V_2$ (resp. $V'_2$) two dimensional hyperbolic.  Thus, by conjugation we may assume that the splitting of $[\Lambda',n',0,\beta']$ is the same as the splitting of $[\Lambda,n,0,\beta]$, i.e. $V_1'=V_1$ and $V_2'=V_2$.  We have $E=E'$ and $\GE=G_{E'}$ and conjugating further we may assume that $\Lambda'_E$ and $\Lambda_E$ lie in the closure of the same chamber of the building of $\GE$.  Moreover, $\Lambda'_E$ is a vertex and $\Lambda_E$ is the barycentre of the chamber. 

Let 
\[w=\begin{pmatrix}0&1&0\\1&0&0\\0&0&1\end{pmatrix}.\]  
We have
\begin{equation*}
J(\beta',\Lambda)=w\begin{pmatrix}
\begin{tikzpicture}[every node/.style={minimum width=1.5em}]
\matrix (m1) [matrix of math nodes]
{ 
\mathfrak{A}_0(\Lambda)^{11} \\
};
\node[scale=1, anchor=north west] (iron) at (m1-1-1.south east) {$\mathfrak{A}_{0}(\Lambda)^{22}$};
\node[scale=1,anchor=east] (iron2) at (iron.west) {$\mathfrak{A}_{\lfloor \frac{r'+1}{2}\rfloor}(\Lambda)^{12}$};
\node[scale=1, anchor= west] (pri) at (m1-1-1.east) {$\mathfrak{A}_{\lfloor \frac{r'+1}{2}\rfloor}(\Lambda)^{12}$};
\draw (m1-1-1.north east) -- (iron2.south east);
\draw (m1-1-1.south west) -- (pri.south east);
\end{tikzpicture}
\end{pmatrix}w \cap G,
\end{equation*}
and
\begin{equation*}\quad J(\beta,\Lambda)=w\begin{pmatrix}
\begin{tikzpicture}[every node/.style={minimum width=1.5em}]
\matrix (m1) [matrix of math nodes]
{ 
\mathfrak{A}_0(\Lambda)^{11} \\
};
\node[scale=1, anchor=north west] (iron) at (m1-1-1.south east) {$\mathfrak{A}_{0}(\Lambda)^{22}$};
\node[scale=1,anchor=east] (iron2) at (iron.west) {$\mathfrak{A}_{\lfloor \frac{r+1}{2}\rfloor}(\Lambda)^{12}$};
\node[scale=1, anchor= west] (pri) at (m1-1-1.east) {$\mathfrak{A}_{\lfloor \frac{r+1}{2}\rfloor}(\Lambda)^{12}$};
\draw (m1-1-1.north east) -- (iron2.south east);
\draw (m1-1-1.south west) -- (pri.south east);
\end{tikzpicture}
\end{pmatrix}w \cap G,
\end{equation*}
where $r'$ (resp. $r$) is minimal such that $[\Lambda,n',r',\beta]$ (resp. $[\Lambda,n,r,\beta]$) is equivalent to a scalar stratum. Thus, as we are now assuming that $[\Lambda,n,n-1,\beta]$ and $[\Lambda',n',n'-1,\beta']$ are not equivalent to scalar simple strata, we have $r'=n'$ and $r=n$.  Furthermore, we have $l(\chi)=l(\pi)$, i.e. $n'/e(\Lambda')=n/e(\Lambda)$.  We let $\kappa''$ be the unique $\beta$-extension to $J(\beta',\Lambda)$ compatible with $\kappa'$ relative to a semisimple stratum $[\Lambda,n'',0,\beta']$ .  By direct computation, $n''=n$.  Hence $J(\beta,\Lambda)=J(\beta',\Lambda)$.  Similarly considerations show that $H(\beta,\Lambda)=H(\beta',\Lambda)$ and $J(\beta,\Lambda')=J(\beta',\Lambda')$.

 As $\xi$ is not cuspidal, it is direct factor of $i_{B(\Lambda'_E)}^{M(\Lambda'_E)}(\widehat{\tau'})$, where we choose $B(\Lambda'_E)$ to be the image of $P(\Lambda_E)$ in $M(\Lambda'_E)$, for some representation $\widehat{\tau'}$ of $T(\Lambda'_E)$. Furthermore, $i^G_B(\widehat{\chi})$ contains $\widehat{\kappa''}\otimes \widehat{\tau'}$ with $\widehat{\kappa''}$ a lift of $\kappa''$ by Lemma \ref{compatbeta} and transitivity of induction. By Theorem \ref{plcovers} $(J,\widehat{\kappa''}\otimes\widehat{\tau'})$ is a $G$-cover of $(T^0,\widehat{\kappa''}_T\otimes\widehat{\tau'})$ relative to $B$ where $\widehat{\kappa''}_T=\Res^J_{T^0}(\widehat{\kappa''})$.  By Blondel  \cite[Theorem 2]{blondel}, $\ind_J^G(\widehat{\kappa''}\otimes\widehat{\tau'})\simeq \Ind_{B^{op}}^G(\ind^T_{T_0}(\widehat{\kappa''}_T\otimes\widehat{\tau'}))$.  By second adjunction of parabolic induction and right adjunction of restriction with compact induction we have
\[\Hom_{T^0}(\widehat{\kappa''}_T\otimes\widehat{\tau'}, r^G_B\circ i^G_B(\widehat{\chi}))\simeq \Hom_G(\ind_J^G(\widehat{\kappa''}\otimes\widehat{\tau'}),i^G_B(\widehat{\chi}))\neq 0.\]
 We have $[r^G_B\circ i^G_B(\widehat{\chi})\mid_{T^0}]=\widehat{\chi}\oplus\widehat{\chi}^{w_x}\mid_{T^0}=\widehat{\chi}\oplus\widehat{\chi}\mid_{T^0}$. Hence $\widehat{\kappa''}_T\otimes\widehat{\tau'}=\Res^T_{T^0}(\widehat{\chi})$.   Similarly if we let $\widehat{\kappa}$ be a lift of $\kappa$, $\widehat{\sigma}$ a lift of $\sigma$, and $\widehat{\kappa}_T=\Res^T_{T^0}(\widehat{\kappa})$ then we have $\widehat{\kappa}_T\otimes \widehat{\sigma}=\Res^J_{T^0}(\widehat{\chi})$.  This implies that we have an equality of semisimple characters $\tau_{\Lambda',\Lambda,\beta'}(\widehat{\theta'})=\widehat{\theta}$ where $\widehat{\theta'}\in\mathcal{C}_{-}(\beta',\Lambda')$ is contained in $\widehat{\kappa'}$ and $\widehat{\theta}\in\mathcal{C}(\beta,\Lambda)$ is contained in $\widehat{\kappa}$.
 
 We let $\widetilde{H}(\beta,\Lambda)$ (resp. $\widetilde{H}^1(\beta',\Lambda')$) denote the compact open subgroup of $\GL_3(F)$ defined in Stevens \cite{St08} which defines $H(\beta,\Lambda)$ (resp. $H(\beta',\Lambda')$) by intersecting with $\U(2,1)(F/F_0)$.  The Iwahori decomposition for $\widetilde{H}^1(\beta',\Lambda')$ gives $\widetilde{H}^1(\beta',\Lambda')=\widetilde{H}^1(\beta',\Lambda')^{-}(\widetilde{H}^1(\beta',\Lambda')\cap \widetilde{M}) \widetilde{H}^1(\beta',\Lambda')^+$ where $\widetilde{H}^1(\beta',\Lambda')^{-}$ denotes the lower triangular unipotent matrices in $\widetilde{H}^1(\beta',\Lambda')$, $\widetilde{H}^1(\beta',\Lambda')^+$ denotes the upper triangular unipotent matrices in $\widetilde{H}^1(\beta',\Lambda')$, and $\widetilde{M}$ the subgroup of diagonal matrices.  As $\widetilde{H}^1(\beta',\Lambda)$ contains  $(\widetilde{H}^1(\beta',\Lambda')\cap \widetilde{M})$ and is contained in $\widetilde{H}^1(\beta',\Lambda')$ we have 
 \[\widetilde{H}^1(\beta',\Lambda')=\widetilde{H}^1(\beta',\Lambda')^{-}(\widetilde{H}^1(\beta',\Lambda')\cap \widetilde{H}^1(\beta',\Lambda)) \widetilde{H}^1(\beta',\Lambda')^+.\] 
 Thus a character of $\widetilde{H}^1(\beta',\Lambda')$ is determined by its values on $\widetilde{H}^1(\beta',\Lambda')^{-}$, $(\widetilde{H}^1(\beta',\Lambda')\cap \widetilde{H}^1(\beta',\Lambda))$, and $\widetilde{H}^1(\beta',\Lambda')^+$.
 
 The semisimple characters $\widehat{\theta}$ and $\widehat{\theta'}$ are equal to the restriction of semisimple characters $\widetilde{\theta}$ and $\widetilde{\theta'}$ of $\GL_3(F)$.  Moreover $\tau_{\Lambda',\Lambda,\beta'}(\widetilde{\theta'})=\widetilde{\theta}$ as $\tau_{\Lambda',\Lambda,\beta'}(\widehat{\theta'})=\widehat{\theta}$. It follows from the decomposition of $\widetilde{H}^1(\beta',\Lambda')$ given above that $\tau_{\Lambda,\Lambda',\beta}(\widetilde{\theta})=\widetilde{\theta'}$; they are both trivial on $\widetilde{H}^1(\beta',\Lambda')^{-}$ and $\widetilde{H}^1(\beta',\Lambda')^+$, and as $\theta'=\tau_{\Lambda,\Lambda',\beta'}(\theta)$ they both agree with $\theta$ on $(\widetilde{H}^1(\beta,\Lambda')\cap \widetilde{H}^1(\beta,\Lambda))=(\widetilde{H}^1(\beta',\Lambda')\cap \widetilde{H}^1(\beta',\Lambda))$.  Hence $\tau_{\Lambda,\Lambda',\beta}(\theta)=\theta'$ by restriction and reduction modulo-$\ell$.  As there is a unique Heisenberg representation containing $\theta'$ we have $\R_{\kappa_m}(\pi)\neq 0$ for some $m\in\{x,y\}$.
 \end{proof}
Lemma \ref{compatbeta} and transitivity of induction $i^G_B(\widehat{\chi})$ contains $\widehat{\kappa}\otimes\delta$ with $\widehat{\kappa}$ the unique $\beta_{M}$-extension of $\tau_{\Lambda^{M},\Lambda,\beta_{M}}(\theta_{M})$ compatible with $\widehat{\kappa}_{M}$.  

\begin{remark}\label{scusplift}
Let $\I_{\kappa'}(\sigma')$ be an irreducible cuspidal representation of $G$ as constructed in Theorem \ref{plcusps}. As a corollary to the proof of Lemma \ref{lemmacus}, we see that $\I_{\kappa'}(\sigma')$ is supercuspidal if and only if $\sigma'$ is supercuspidal.  Hence all supercuspidal representations of $G$ lift by Section \ref{cuspU11U21}.
\end{remark}

\begin{lemma}\label{lemma65}
Suppose that $i^G_B(\chi)$ is reducible with irreducible subrepresentation $\pi_1$ and quotient $\pi_2=i^G_B(\chi)/\pi_1$. If $\Sigma$ is a maximal cuspidal subquotient of $\R_{\kappa_m}(i^G_B(\chi))$, i.e. all subquotients of $\R_{\kappa_m}(i^G_B(\chi))$ not contained in $\Sigma$ are not cuspidal, then $\I_{\kappa_m}(\Sigma)$ is a subrepresentation of $\pi_c$. \end{lemma}

\begin{proof}
Let $\Sigma$ be a maximal cuspidal subquotient of $\R_{\kappa_m}(i^G_B(\chi))$.  By Lemma \ref{lemma64}, $\R_{\kappa_{m}}(\pi_1)$ and $\R_{\kappa_{m}}(\pi_2)$ are non-zero and must contain non-cuspidal subquotients as $\pi_1$ and $\pi_2$ are not cuspidal.  However by Theorem \ref{SZadapt} and Section \ref{cuspU11U21}, there are only two non-cuspidal subquotients of $\R_{\kappa_m}(i^G_B(\chi))$.  Thus both $\R_{\kappa_{m}}(\pi_1)$ and $\R_{\kappa_{m}}(\pi_2)$ must have a single non-cuspidal irreducible subquotient, say $\rho_1$ and $\rho_2$. 

If $\R_{\kappa_m}(\pi_1)\neq \rho_1$ then $\R_{\kappa_m}(\pi_1)$ has an irreducible cuspidal subrepresentation or an irreducible cuspidal quotient. If $\R_{\kappa_m}(\pi_1)$ has an irreducible cuspidal subrepresentation $\sigma$ then, by adjointness of $\R_{\kappa_m}$ and $\I_{\kappa_m}$, $\I_{\kappa_m}(\sigma)$ is an irreducible cuspidal subrepresentation of $\pi_1$ contradicting the irreducibility and non-cuspidality of $\pi_1$. If $\R_{\kappa_m}(\pi_1)$ has an irreducible cuspidal quotient $\sigma$ then $\R_{\widetilde{\kappa_m}}(\widetilde{\pi_1})$ has a cuspidal subrepresentation $\widetilde{\sigma}$ by Lemma \ref{contracommutes}. Thus $\I_{\widetilde{\kappa_m}}(\widetilde{\sigma})$ is an irreducible cuspidal subrepresentation of $\widetilde{\pi_1}$ by adjointness.  Hence $\I_{\kappa_m}({\sigma})$ is an irreducible cuspidal quotient of $\pi_1$ by Lemma \ref{contracommutes} contradicting the irreducibility and non-cuspidality of $\pi_1$.  Thus $\R_{\kappa_m}(\pi_1)=\rho_1$.

Similarly, if $\R_{\kappa_m}(\pi_2)$ has an irreducible cuspidal quotient $\sigma$, then $\I_{\kappa_m}(\sigma)$ is an irreducible cuspidal quotient of $\pi_2$.  Hence $\I_{\kappa_m}(\sigma)$ is a quotient of $i^G_B(\chi)$ contradicting the cuspidality of $\I_{\kappa_m}(\sigma)$.  Hence $\R_{\kappa_m}(\pi_2)$ can have no cuspidal quotients. Hence, by Section \ref{cuspU11U21}, Lemma \ref{theorem2} and Theorem \ref{SZadapt}, $\Sigma$ is a subrepresentation of $\R_{\kappa_m}(\pi_2)$. , Note that, as Theorem \ref{SZadapt} only gives us an isomorphism in the Grothendieck group of finite length representations of $\M(\Lambda_E)$ we have used that $\Sigma$ is irreducible by Section \ref{cuspU11U21} in the skew semisimple non-scalar case to imply it is a subrepresentation of $\R_{\kappa_m}(\pi_2)$, in all other cases we twist by a character (if necessary) and use Lemma \ref{theorem2}. By reciprocity, $\I_{\kappa_m}(\Sigma)$ is a subrepresentation of $\pi_2$.  
\end{proof}

By Blondel \cite[Theorem 2]{blondel} and Lemma \ref{plcovers}, $\I_{\kappa}(\sigma)\simeq \Ind^G_{B^{op}}(\ind_{T^0}^T(\kappa_T\otimes\sigma)).$ By second adjunction, (\emph{cf.} Dat \cite[Corollaire $3.9$]{finitude}) $\Hom_G(\Ind^G_{B^{op}}(\ind_{T^0}^T(\kappa_T\otimes\sigma)),\pi)\simeq \Hom_T(\ind_{T^0}^T(\kappa_T\otimes\sigma),r^G_B(\pi))$.  By Clifford theory, the irreducible quotients of $\ind_{T^0}^T(\kappa_T\otimes\sigma)$ are all the twists of $\chi$ by an unramified character. Hence  $\pi$ is an irreducible quotient of $\I_{\kappa}(\sigma)$ if and only if it is an irreducible quotient of $i^G_B(\chi\psi)$ for some unramified character $\psi$ of $T$.

The $R$-type $(J, \lambda)$ is quasi-projective by Theorem \ref{plcusps}, hence a simple module of $\mathcal{H}(G,\lambda)$ corresponds to an irreducible quotient of $i^G_B(\chi\psi)$ for some unramified character $\psi$ by the bijection of Theorem \ref{qplemma}.  If $i^G_B(\chi\psi)$ is reducible with proper quotient $\pi$, then  the Jacquet module of $\pi$ is one-dimensional by the geometric lemma,.  Hence, by Theorem \ref{jacquetcommutes}, $\pi$ must correspond to a character of $\mathcal{H}(G,\lambda)$ under the bijection of Theorem \ref{qplemma} and all characters of $\mathcal{H}(G,\lambda)$ must correspond to a proper quotient of a reducible principle series representation $i^G_B(\chi\psi)$ with $\psi$ an unramified character of $T$.

\begin{lemma}\label{semisimple}
Suppose $\ell\neq 2$ and $\ell\mid q-1$.  Then $i^G_B(\chi)$ is semisimple.
\end{lemma}

\begin{proof}
If $i^G_B(\chi)$ is irreducible then it is semisimple, so suppose $i^G_B(\chi)$ is reducible.  If $i^G_B(\chi)$ has a cuspidal subquotient it is of the form $\I_{\kappa_m}(\sigma)$ for $m\in\{x,y\}$ and $\sigma$ an irreducible cuspidal representation of $\M(\Lambda_E^x)$ by Theorem \ref{plcusps}.  By Theorem \ref{SZadapt} $\R_{\kappa_x}(i^G_B(\chi))=i^{\M(\Lambda^x_E)}_{B(\Lambda^x_E)}(\R_{\kappa_T}(\chi))$ and $\R_{\kappa_x}(\I_{\kappa_x}(\sigma))=\sigma$, by Theorem \ref{pl=lz} and Lemma \ref{lemma111}.  Hence, by exactness, $\sigma$ is a cuspidal subquotient of $i^{\M(\Lambda^x_E)}_{B(\Lambda^x_E)}(\R_{\kappa_T}(\chi))$.  However, by Section \ref{cuspU11U21} when $\ell\mid q-1$ no such cuspidal subquotients exist hence $i^G_B(\chi)$ has no cuspidal subquotients.  Thus, by exactness of the Jacquet functor and the geometric lemma, $i^G_B(\chi)$ has length two.   When $\ell\neq 2$ and $\ell\mid q-1$ there are four characters of $\mathcal{H}(G,\lambda)$, yet only two reducibility points.  Hence these two reducible principle series representations must both have two non-isomorphic irreducible quotients and must be semisimple.
\end{proof}

\begin{remark}
If $\chi^{-1}=\chi$ then $i_B^G(\chi)$ is self contragredient and there is a simple proof of Lemma \ref{semisimple} using the contragredient representation and avoiding the use of covers or second adjunction.
\end{remark}

\begin{lemma}\label{lemma611}
Let $\ell\mid q+1$.  Then the unique irreducible quotient of $i^G_B(\chi)$ is isomorphic to the unique irreducible subrepresentation.
\end{lemma}

\begin{proof}
Let $\pi$ denote the unique irreducible quotient of $i^G_B(\chi)$. When $\ell\mid q+1$ there is only one character of  $\mathcal{H}(G,\kappa\otimes\sigma)$. Hence $\pi$ corresponds to the unique character of $\mathcal{H}(G,\lambda)$.  Hence, if $\mathcal{V}$ is the space of $\pi$, $\R_{\kappa}(\mathcal{V})$ is one dimensional and the action of $J$ is given by $\sigma$. As $\delta_B$ is trivial, the contragredient commutes with parabolic induction; we have $\left(i_B^G(\chi)\right)^{\sim} \simeq i_B^G\left(\widetilde{\chi}\right)$. Furthermore, $\widetilde{\chi}=\chi^{-1}$ where $\chi^{-1}$ is the character defined by, for all $x\in F^\times$,   $\chi^{-1}(x)=\chi(x^{-1})$. The character $\chi^{-1}$ is not regular and similar arguments, given for $i_B^G(\chi)$, apply to $i_B^G(\chi^{-1})$.  We find that $i_B^G(\chi^{-1})$ has a unique irreducible quotient $\rho$ which corresponds to the unique character of $\mathcal{H}(G,\widetilde{\lambda})$ under the bijection of Theorem \ref{qplemma}. 
As the contragredient is contravariant and exact, $\widetilde{\rho}$ is a subrepresentation of $i_B^G(\chi)$.  By Lemma \ref{contracommutes}, we have  $\left(\R_{\widetilde{\kappa}}(\rho)\right)^{\sim}\simeq \R_{\kappa}(\widetilde{\rho})$ which is one dimensional and hence must be isomorphic to $\sigma$.  Hence $\widetilde{\rho}$ is irreducible and isomorphic to $\pi$.  Thus $\pi$ appears twice in the composition series of $i_B^G(\chi)$ as the unique irreducible quotient and as the unique irreducible subrepresentation. 
\end{proof}

\begin{remark}
If $\ell\neq 3$ and $\ell\mid q^2-q+1$ then similar counting arguments show that the unique irreducible subrepresentation is not isomorphic to the unique irreducible quotient.  However, in these cases we find out more information later so this argument is not necessary.
\end{remark}

\subsection{On the unramified principal series}
\subsubsection{Decomposition of $i_{B}^G(\delta_B^{-\frac{1}{2}})$ and $i^G_B(\delta_B^{\frac{1}{2}})$}
In all cases of coefficient field, the space of constant functions form an irreducible subrepresentation of $i_{B}^G(\delta_B^{-\frac{1}{2}})$ isomorphic to $1_G$.  We let $\St_G$ denote the quotient of $i_{B}^G(\delta_B^{-\frac{1}{2}})$ by $1_G$. Parabolic induction preserves finite length representations, hence $\St_G$ has an irreducible quotient $\nu_G$.   By the geometric lemma, $\left[r^G_B\circ i_{B}^G(\delta_B^{-\frac{1}{2}})\right]\simeq \delta_B^{-\frac{1}{2}}\oplus(\delta_B^{-\frac{1}{2}})^{w_x}$.  Considering $\delta_B$ as a character of $E^\times$, we have $(\delta_B^{-\frac{1}{2}})^{w_x}(x)=\delta_B^{-\frac{1}{2}}(\overline{x}^{-1})=\delta_B^{\frac{1}{2}}(x),$
as $\delta_B(x)=\delta_B(\overline{x})$.  Thus $\left[r^G_B\circ i_{B}^G(\delta_B^{-\frac{1}{2}})\right]=\delta_B^{-\frac{1}{2}}\oplus\delta_B^{\frac{1}{2}}$.  We have $r^G_B(1_G)=\delta_B^{-\frac{1}{2}}$, thus $r^G_B(\St_G)=\delta_B^{\frac{1}{2}}$ by exactness of the Jacquet functor.  A quotient of a parabolically induced representation has nonzero Jacquet module, hence $r^G_B(\nu_G)=\delta_B^{\frac{1}{2}}$.    Thus any other composition factors which occur in $i_{B}^G(\delta_B^{-\frac{1}{2}})$ must be cuspidal.  

\begin{thm}\label{urprincipaldecomp}~
\begin{enumerate}
\item If $\ell\nmid (q-1)(q+1)(q^2-q+1)$ then $i^G_B(\delta_B^{-\frac{1}{2}})$ has length two with unique irreducible subrepresentation $1_G$ and unique irreducible quotient $\St_G$.
\item If $\ell\neq 2$ and $\ell\mid q-1$ then $i^G_B(\delta_B^{-\frac{1}{2}})=1_G\oplus \St_G$ is semisimple of length two.
\item If $\ell\neq 3$ and $\ell\mid q^2-q+1$ then $i^G_B(\delta_B^{-\frac{1}{2}})$ has length three with unique cuspidal subquotient $\I_{\Lambda_x}(\overline{\tau}^+(\overline{1}))$.  The unique irreducible quotient $\nu_G$ is not a character.
\item If $\ell\neq 2$ and $\ell\mid q+1$ or if $\ell=2$ and $4\mid q+1$,  then $i^G_B(\delta_B^{-\frac{1}{2}})$ has length six with $1_G$ appearing as the unique subrepresentation and the unique quotient and four cuspidal subquotients.  Let $\pi$ be a maximal proper submodule of $\St_G$.  Then $\pi\simeq \rho\oplus \I_{\Lambda_y}(\overline{\sigma}(\overline{1})\otimes\overline{1})$ where $\rho$ is of length three with unique 
irreducible subrepresentation and unique irreducible quotient, both of which are isomorphic to $\I_{\Lambda_x}(\overline{\nu}(\overline{1}))$, and remaining subquotient isomorphic to $\I_{\Lambda_x}(\overline{\sigma}(\overline{1}))$.
\item  If $\ell=2$ and $4\mid q-1$, then $i^G_B(\delta_B^{-\frac{1}{2}})$ has length five with unique irreducible subrepresentation and unique irreducible quotient both isomorphic to $1_G$. Let $\pi$ be a maximal proper submodule of $\St_G$.  Then $\pi\simeq \I_{\Lambda_x}(\overline{\nu}(\overline{1}))\oplus\I_{\Lambda_x}(\overline{\tau}^{+}(\overline{\chi}))\oplus\I_{\Lambda_y}(\overline{\sigma}(\overline{1})\otimes\overline{1})$.
\end{enumerate}
\end{thm}

\begin{proof}
By Theorem \ref{plcusps} and Lemma \ref{lemmacus}, if $i^G_B(\delta_B^{-\frac{1}{2}})$ has a cuspidal subquotient $\pi$ then $\pi\simeq \I_{\Lambda_w}(\sigma)$ for $w\in\{x,y\}$ and $\sigma$ an irreducible cuspidal representation of $\P(\Lambda_w)/\P_1(\Lambda_w)$.

If $\Sigma_w$ is a maximal cuspidal subquotient of $\R_{\Lambda_w}(i^G_B(\delta_B^{-\frac{1}{2}}))$ then $\I_{\Lambda_w}(\Sigma_w)$ is a subrepresentation of $\St_G$ by Lemma \ref{lemma65}.  Thus, we have an exact sequence
\[0\rightarrow \I_{\Lambda_x}(\Sigma_x)\oplus\I_{\Lambda_y}(\Sigma_y)\rightarrow \St_G\rightarrow \nu_G\rightarrow 0.\]
By exactness and Section \ref{cuspU11U21}, we obtain composition series of $\I_{\Lambda_x}(\Sigma_x)$ and of $\I_{\Lambda_y}(\Sigma_y)$.

If $\ell\nmid (q-1)(q+1)(q^2-q+1)$ or $\ell\neq 2$ and $\ell\mid q-1$ then $\R_{\Lambda_x}(i^G_B(\delta_B^{-\frac{1}{2}}))$ and $\R_{\Lambda_y}(i^G_B(\delta_B^{-\frac{1}{2}}))$ are of length two with no cuspidal subquotients, by Theorem \ref{plcusps} and Lemma \ref{lemmacus}.  Hence $i^G_B(\delta_B^{-\frac{1}{2}})$ has no cuspidal subquotients as $\R_{\Lambda_w}(\I_{\Lambda_w}(\sigma))\simeq \sigma$ is cuspidal by Lemma \ref{lemma111}.  By the geometric lemma, $i^G_B(\delta_B^{-\frac{1}{2}})$ is of length two with $1_G$ as an irreducible subrepresentation and $\St_G$ as an irreducible quotient.  By second adjunction,
\[\Hom_G(i^G_B(\delta_B^{-\frac{1}{2}}),1_G)\simeq \Hom_T(\delta_B^{-\frac{1}{2}},1_T).\]
 The character $\delta_B^{-\frac{1}{2}}$ is nontrivial when $\ell\nmid (q-1)(q+1)(q^2-q+1)$ and trivial when $\ell\mid q-1$.  Hence $1_G$ is a direct factor when $\ell\neq 2$ and $\ell\mid q-1$ and $i^G_B(\delta_B^{-\frac{1}{2}})$ is semisimple, and $i^G_B(\delta_B^{-\frac{1}{2}})$ is non-split when $\ell\nmid (q-1)(q+1)(q^2-q+1)$.
 
In all other cases, $i^G_B(\delta_B^{-\frac{1}{2}})$ has cuspidal subquotients.  Thus $1_G$ cannot be a direct factor.  Therefore $i^G_B(\delta_B^{-\frac{1}{2}})$  has a unique irreducible quotient $\nu_G$ and a unique irreducible subrepresentation $1_G$.  When $\ell\mid q+1$ the unique irreducible quotient is isomorphic to the unique irreducible subrepresentation by Lemma \ref{lemma611}, hence $\nu_G\simeq 1_G$.  When $\ell\neq 3$ and $\ell\mid q^2-q+1$ then $\R_{\Lambda_y}(i^G_B(\delta_B^{-\frac{1}{2}})$ has non-cuspidal subquotients $1_{M_y}$ and $\St_{M_y}$.  By exactness, $\R_{\Lambda_y}(\nu_G)\simeq \St_{M_y}$ hence $1_G$ is not isomorphic to $\nu_G$ which is not a character.
\end{proof}

Note that $i^G_B(\delta_B^{\frac{1}{2}})\simeq i^G_B(\delta_B^{-\frac{1}{2}})^{\sim}$, hence decompositions of $i^G_B(\delta_B^{\frac{1}{2}})$ can be obtained from Theorem \ref{urprincipaldecomp}.

\subsubsection{Decomposition of unramified $i^G_B(\eta\delta_B^{-\frac{1}{4}})$ and $i^G_B(\eta\delta_B^{\frac{1}{4}})$}
Let $\eta$ be the unique unramified character of $F^\times$ extending $\omega_{F/F_0}$.  When $\ell\mid q+1$ we have $\omega_{F/F_0}\delta_B^{\frac{1}{4}}=\omega_{F/F_0}\delta_B^{-\frac{1}{4}}=\delta_B=1$ hence we refer to Theorem \ref{urprincipaldecomp}.  When $\ell\mid q^2+q+1$ we have $\delta_B^{-\frac{1}{2}}=\eta\delta_B^{\frac{1}{4}}$ and $\delta_B^{\frac{1}{2}}=\eta\delta_B^{-\frac{1}{4}}$, hence once more we refer to Theorem \ref{urprincipaldecomp}.   When $\ell\mid q-1$, $\delta_B$ is trivial; hence $\eta\delta_B^{-\frac{1}{4}}=\eta\delta_B^{\frac{1}{4}}=\eta$.  Thus $i^G_B(\eta)$ is self-contragredient. By Lemma \ref{semisimple}, $i^G_B(\eta)$ has length two and is semisimple.

\subsection{Cuspidal subquotients of the ramified level zero principal series}
We describe the reducible principal series $i^G_B(\chi)$ which have length greater than two when $\chi$ is a level zero character of $T$ which does not factor through the determinant map.  We twist by a character that factors through the determinant map so that we can assume $\chi_2=1$. Then $\chi^{q+1}=1$ and $\chi=\overline{\psi}\circ\xi$ for $\overline{\psi}$ a non-trivial character of $k_F^1$. 

When $\ell\nmid q+1$, because $\R_{\Lambda_x}(i^G_B(\chi))$ and $\R_{\Lambda_y}(i^G_B(\chi))$ have no cuspidal subquotients, 
$i^G_B(\chi)$ is of length two.  
\begin{thm}\label{RamifiedPrincseries}
Let $\ell\mid q+1$. The representation  $i_B^G(\chi)$ has length four with a unique irreducible subrepresentation and a unique irreducible quotient, and cuspidal subquotient isomorphic to $\I_{\Lambda_x}(\overline{\sigma}(\overline{\psi},\overline{\psi},\overline{1}))\oplus\I_{\Lambda_y}(\overline{\sigma}(\overline{\psi})\otimes \overline{1})$. Furthermore, the unique irreducible subrepresentation is isomorphic to the unique irreducible quotient. \end{thm}

\begin{proof}
The proof is similar to the proof of Theorem \ref{urprincipaldecomp}.
\end{proof}

\section{Cuspidal subquotients of positive level principal series}\label{penultimate}
In this section, suppose that $\chi_1$ is a positive level character of $F^\times$ trivial on $F_0^\times$ and $\chi$ is the character of $T$ given by $\chi_1$ and $\chi_2=1$.  We assume we are in the same setting as Section \ref{sect64} with $(T^0,\lambda_T)$ an $R$-type contained in $\chi$, $(J,\lambda)$ a $G$-cover of $(T^0,\lambda_T)$ relative to $B$ with $\lambda=\kappa\otimes\sigma$, and $(\kappa_m,\Lambda^m)$ compatible with $(\kappa,\Lambda)$ for $m\in\{x,y\}$. We have $\M(\Lambda^m_E)\simeq\U(1,1)(k_F/k_0)\times \U(1)(k_F/k_0)$.  When $\ell\nmid q+1$, there are no cuspidal subquotients of $\U(1,1)(k_F/k_0)$ and hence no cuspidal subquotients of $i^G_B(\chi)$ by Lemma \ref{lemmacus}.  Thus it remains to look at the case when $\ell\mid q+1$.   Let $\overline{\psi}=(\chi\kappa_T^{-1})^{T^1}$ and $\overline{\chi}$ the character of $k_F^1$ such that $\overline{\psi}=\overline{\chi}\circ \xi$. 

\begin{thm}\label{poslevramdecomp}
Suppose $\ell\mid q+1$.    The representation $i^G_B(\chi)$ has length four with unique irreducible subrepresentation and unique irreducible quotient which are isomorphic, and cuspidal subquotient isomorphic to $\I_{\kappa_x}(\sigma(\overline{\chi})\otimes \overline{1})\oplus \I_{\kappa_y}(\sigma(\overline{\chi})\otimes \overline{1}).$ 
\end{thm}

\begin{proof} The proof is similar to the proof of Theorem \ref{urprincipaldecomp}.
\end{proof}

\section{Supercuspidal support}

\begin{thm}\label{scuspsupp}
Let $G$ be an unramified unitary group in three variables and $\pi$ an irreducible $\ell$-modular representation of $G$.  Then the supercuspidal support of $\pi$ is unique up to conjugacy. 
\end{thm}

\begin{proof}
Suppose $\pi$ is not cuspidal.  Then the supercuspidal support of $\pi$ is equal to the cuspidal support of $\pi$ and is thus unique up to conjugacy.  If $\pi$ is cuspidal non-supercuspidal then it appears in one of the decompositions given in Theorems \ref{urprincipaldecomp}, \ref{RamifiedPrincseries}, and \ref{poslevramdecomp}, or is a twist of such a representation by a character that factors through the determinant map, and we see that the supercuspidal support of $\pi$ is unique up to conjugacy.
\end{proof}

\bibliographystyle{plain}
\bibliography{typesandmodrep}

\end{document}